\documentclass{amsart}
\usepackage{amssymb}
\usepackage{amsmath}
\usepackage{amsfonts}

\newtheorem{theorem}{Theorem}
\theoremstyle{plain}

\newtheorem{corollary}{Corollary}

\newtheorem{definition}{Definition}
\newtheorem{example}{Example}

\newtheorem{proposition}{Proposition}
\newtheorem{remark}{Remark}

\numberwithin{equation}{section}
\input{tcilatex}

\begin{document}
\title[Non-Newtonian Approach to $C^{\ast }-$algebras]{Non-Newtonian
Approach to $C^{\ast }-$Algebras}
\author{Nilay DE\u{G}\.{I}RMEN}
\address[N. De\u{g}irmen]{Ondokuz Mayis University Faculty of Art and
Sciences Deparment of Mathematics Samsun, Turkey}
\email[N. De\u{g}irmen]{nilay.sager@omu.edu.tr}

\begin{abstract}
In this paper, we study the non-Newtonian version of $C^{\ast }-$algebras.
Further, we generalize some results which hold for the classical $C^{\ast }-$%
algebras. We also discuss some illustrative examples to show accuracy and
effectiveness of the new findings. If we take the identity function $I$
instead of the generators $\alpha $ and $\beta $ in the construction of the
set $%
\mathbb{C}
\left( N\right) $, then non-Newtonian $C^{\ast }-$algebras turn into the
classical $C^{\ast }-$algebras, so our results are stronger than some
knowledge and facts in the most existing literature.
\end{abstract}

\maketitle

\section{Introduction and Preliminaries}

Grossman and Katz \cite{1} introduced non-Newtonian calculus as an
alternative to classical calculus in the period from 1967 till 1972. They
defined an infinite family of calculus which includes some special calculi
such as geometric calculus, harmonic calculus, bigeometric calculus,
anageometric calculus. Since this structure has been developed very fast in
recent times due to excellent applications in different fields of
mathematical sciences including engineering, economy, biology, approximation
theory, probability theory, weighted calculus etc., it has attracted
considerable interest from many authors. There is a counterpart in each
member of non-Newtonian calculus class of all structures used in classical
calculus. So, there are many studies in the literature that extend some
concepts in classical calculus to those in non-Newtonian calculus. Tekin and
Ba\c{s}ar \cite{2} constructed the $\ast -$complex field of non-Newtonian
complex numbers which is of great importance for our this work and proved
some basic properties. For other related studies on non-Newtonian calculus,
we recommend \cite{3,4,5,6,7,8}.

The theory of $C^{\ast }-$algebra is one of the most popular and important
research areas in operator theory and functional analysis, which was studied
by many mathematicians in a long period of time. It is well known that some
related applications have been widely studied in various disciplines of
mathematics, theoretical physics and noncommutative geometry. Some of areas
studied on can be expressed as the theory of operators on a Hilbert space,
statistical physics and quantum field theory, the theory of representations
of groups and symmetric algebras and the theory of dynamical systems.

Now, we recollect some basic definitions, notations and results that will be
used in other sections.

A complete ordered field is called an arithmetic if its realm is a subset of 
$%
\mathbb{R}
.$ A generator is a one-to-one function whose domain $%
\mathbb{R}
$ and whose range is a subset of $%
\mathbb{R}
.$ Let $\alpha $ be a generator with range $A.$ We denote by $%
\mathbb{R}
_{\alpha }$ the range of generator $\alpha .$ Also, the elements of $%
\mathbb{R}
_{\alpha }$ (or $%
\mathbb{R}
\left( N\right) _{\alpha }$) are called non-Newtonian real numbers. Taking $%
\alpha =I$, the generator $\alpha $ generates the classical arithmetic and $%
\mathbb{R}
_{\alpha }=%
\mathbb{R}
$. Let $\alpha $ and $\beta $ be arbitrarily chosen generators which image
the set $%
\mathbb{R}
$ to $A$ and $B$ respectively and $\ast -(^{\prime \prime }$star-$^{\prime
\prime })$ calculus also be the ordered pair of arithmetics $\left( \alpha -%
\text{arithmetic},\beta -\text{arithmetic}\right) $. The following notations
will be used. All definitions given for $\alpha -$arithmetic are also valid
for $\beta -$arithmetic.%
\begin{equation*}
\begin{tabular}{lll}
& $\alpha -$arithmetic & $\beta -\text{arithmetic}$ \\ 
Realm & \multicolumn{1}{c}{$A\left( =%
\mathbb{R}
_{\alpha }=%
\mathbb{R}
\left( N\right) _{\alpha }\right) $} & \multicolumn{1}{c}{$B\left( =%
\mathbb{R}
_{\beta }=%
\mathbb{R}
\left( N\right) _{\beta }\right) $} \\ 
Summation & \multicolumn{1}{c}{$y\overset{.}{+}z=\alpha \left\{ \alpha
^{-1}\left( y\right) +\alpha ^{-1}\left( z\right) \right\} $} & 
\multicolumn{1}{c}{$\ddot{+}$} \\ 
Subtraction & \multicolumn{1}{c}{$y\overset{.}{-}z=\alpha \left\{ \alpha
^{-1}\left( y\right) -\alpha ^{-1}\left( z\right) \right\} $} & 
\multicolumn{1}{c}{$\ddot{-}$} \\ 
Multiplication & \multicolumn{1}{c}{$y\overset{.}{\times }z=\alpha \left\{
\alpha ^{-1}\left( y\right) \times \alpha ^{-1}\left( z\right) \right\} $} & 
\multicolumn{1}{c}{$\ddot{\times}$} \\ 
Division & \multicolumn{1}{c}{$y\overset{.}{/}z=\frac{y}{z}\alpha =\alpha
\left\{ \frac{\alpha ^{-1}\left( y\right) }{\alpha ^{-1}\left( z\right) }%
\right\} ~\left( z\neq \overset{.}{0}\right) $} & \multicolumn{1}{c}{$\ddot{/%
}~$} \\ 
Ordering & \multicolumn{1}{c}{$y\overset{.}{\leq }z\Longleftrightarrow
\alpha ^{-1}\left( y\right) \leq \alpha ^{-1}\left( z\right) $} & 
\multicolumn{1}{c}{$\overset{..}{\leq }$}%
\end{tabular}%
\end{equation*}

There are the following three properties for the isomorphism from $a-$%
arithmetic to $\beta -$arithmetic that is the unique function $\imath $%
(iota).

1. $\imath $\ is one-to-one.

2. $\imath $\ is on $A$\ and onto $B$.

3. For all $u$,$v\in A$,%
\begin{eqnarray*}
\ \iota \left( u\dot{+}v\right) &=&\iota \left( u\right) \ddot{+}\iota
\left( v\right) ,~\iota \left( u\dot{-}v\right) =\iota \left( u\right) \ddot{%
-}\iota \left( v\right) , \\
\iota \left( u\dot{\times}v\right) &=&\iota \left( u\right) \ddot{\times}%
\iota \left( v\right) ,~\iota \left( u\dot{/}v\right) =\iota \left( u\right) 
\ddot{/}\iota \left( v\right) ,v\neq \dot{0} \\
u &\dot{<}&v\Longleftrightarrow \iota \left( u\right) \ddot{<}\iota \left(
v\right) , \\
\iota \left( u\right) &=&\beta \left\{ \alpha ^{-1}\left( u\right) \right\} ,
\end{eqnarray*}%
Also, for every integer $n$, we set $\iota \left( \dot{n}\right) =\ddot{n}$.

A $\alpha -$positive number is a number $x$ with $\overset{.}{0}\overset{.}{<%
}x$ and a $\alpha -$negative number is a number with $x\overset{.}{<}\overset%
{.}{0}.$ $\alpha -$zero and $\alpha -$one numbers are denoted by $\overset{.}%
{0}=\alpha \left( 0\right) $ and $\overset{.}{1}=\alpha \left( 1\right) ,$
and the set of $\alpha -$positive numbers is denoted by $%
\mathbb{R}
_{\alpha }^{+}$ (or $%
\mathbb{R}
\left( N\right) _{\alpha }^{+}$). Also, $\alpha \left( -p\right) =\alpha
\left\{ -\alpha ^{-1}\left( \overset{.}{p}\right) \right\} =\overset{.}{-}%
\overset{.}{p}$ for all $p\in 
\mathbb{Z}
^{+}.$ An open interval on $%
\mathbb{R}
_{\alpha }$ for $a,b\in 
\mathbb{R}
_{\alpha }$ with $a\overset{.}{<}b$ is represented by $\left( a,b\right)
_{N}=\left\{ x\in 
\mathbb{R}
_{\alpha }:a\overset{.}{<}x\overset{.}{<}b\right\} .$ The $\alpha -$absolute
value of $x\in A$ is defined by%
\begin{equation*}
\overset{.}{\mid }x\overset{.}{\mid }=\left\vert x\right\vert _{\alpha
}=\left\{ 
\begin{array}{c}
x,\text{ \ \ \ \ \ \ \ \ \ \ }if\text{ \ }\overset{.}{0}\overset{.}{<}x \\ 
\overset{.}{0},\text{ \ \ \ \ \ \ \ \ \ }if\text{ \ }\overset{.}{0}=x \\ 
\overset{.}{0}\overset{.}{-}x,\text{ \ \ \ \ }if\text{ \ \ }x\overset{.}{<}%
\overset{.}{0}%
\end{array}%
\right. .
\end{equation*}

Let $\overset{..}{b}\in B\subseteq 
\mathbb{R}
$. Then, the number $\overset{..}{b}\overset{..}{\times }\overset{..}{b}$ is
called the $\beta -$ square of $\overset{..}{b}$ and is denoted by $b^{%
\overset{..}{2}}.$ Let $\overset{..}{b}$ be a nonnegative number in $B$.
Then, $\beta \left[ \sqrt{\beta ^{-1}\left( \overset{..}{b}\right) }\right] $
is called the $\beta -$ square root of $\overset{..}{b}$ and is denoted by $%
\sqrt[..]{\overset{..}{b}}$ \cite{1,3}.

Let $\left( a_{n}\right) \subset 
\mathbb{R}
_{\alpha }.$ An infinite sum%
\begin{equation*}
a_{1}\overset{.}{+}a_{2}\overset{.}{+}...\overset{.}{+}a_{n}\overset{.}{+}%
...=_{\alpha }\underset{n=1}{\overset{\infty }{\sum }}a_{n}
\end{equation*}%
is called the non-Newtonian real number series or $\alpha -$series. If $%
_{\alpha }\underset{n=1}{\overset{\infty }{\sum }}a_{n}^{{}}$ is
non-Newtonian real number series, then a sequence $\left\{ S_{m}\right\} $
with the general term $S_{m}=_{\alpha }\underset{n=1}{\overset{m}{\sum }}%
a_{n}^{{}}$ is called non-Newtonian partial sums sequence of the series $%
_{\alpha }\underset{n=1}{\overset{\infty }{\sum }}a_{n}^{{}}.$ If the
sequence $\left\{ S_{m}\right\} $ is a $\alpha -$convergent, then it is said
that the series $_{\alpha }\underset{n=1}{\overset{\infty }{\sum }}%
a_{n}^{{}} $ is $\alpha -$convergent. If $^{\alpha }\underset{m\rightarrow
\infty }{\lim }S_{m}=S,$ then it is written $_{\alpha }\underset{n=1}{%
\overset{\infty }{\sum }}a_{n}^{{}}=S.$ If the limit $^{\alpha }\underset{%
m\rightarrow \infty }{\lim }S_{m}$ is not available or equal to $-\infty $
or $+\infty ,$ then it is said that the series $_{\alpha }\underset{n=1}{%
\overset{\infty }{\sum }}a_{n}^{{}}$ is $\alpha -$divergent \cite{7}.

The concepts of $\alpha -$convergent sequence, non-Newtonian metric space,
non-Newtonian completeness, non-Newtonian upper bound, non-Newtonian
supremum, non-Newtonian open set, non-Newtonian closed set are discussed in
detail in \cite{1,2,3,7,8}.

Let $\overset{.}{a}\in \left( A,\overset{.}{+},\overset{.}{-},\overset{.}{%
\times },\overset{.}{/},\overset{.}{\leq }\right) $ and $\overset{..}{b}\in
\left( B,\overset{..}{+},\overset{.}{-},\overset{..}{\times },\overset{..}{/}%
,\overset{..}{\leq }\right) $ be arbitrarily chosen elements from
corresponding arithmetics. Then, the ordered pair $\left( \overset{.}{a},%
\overset{..}{b}\right) $ is called as a $\ast -$ point. The set of all $\ast
-$ points is called the set of $\ast -$complex numbers (non-Newtonian
complex numbers) and is denoted by $%
\mathbb{C}
^{\ast }$ or $%
\mathbb{C}
\left( N\right) $ that is,%
\begin{equation*}
\mathbb{C}
\left( N\right) =\left\{ \left( \overset{.}{a},\overset{..}{b}\right) :\text{
}\overset{.}{a}\in A\subseteq 
\mathbb{R}
,\overset{..}{b}\in B\subseteq 
\mathbb{R}
\right\} .
\end{equation*}%
The set $%
\mathbb{C}
\left( N\right) $ forms a field and a Banach space with the algebraic
operations $\oplus $ and $\odot $ defined on $%
\mathbb{C}
\left( N\right) $ and the norm $\overset{..}{\parallel }.\overset{..}{%
\parallel }$ defined by%
\begin{eqnarray*}
\oplus &:&%
\mathbb{C}
\left( N\right) \times 
\mathbb{C}
\left( N\right) \rightarrow 
\mathbb{C}
\left( N\right) , \\
\left( z^{\ast },w^{\ast }\right) &\rightarrow &z^{\ast }\oplus w^{\ast
}=\left( \overset{.}{a_{1}},\overset{..}{b_{1}}\right) \oplus \left( \overset%
{.}{a_{2}},\overset{..}{b_{2}}\right) =\left( \overset{.}{a_{1}}\overset{.}{+%
}\overset{.}{a_{2}},\overset{..}{b_{1}}\overset{..}{+}\overset{..}{b_{2}}%
\right) , \\
\odot &:&%
\mathbb{C}
\left( N\right) \times 
\mathbb{C}
\left( N\right) \rightarrow 
\mathbb{C}
\left( N\right) , \\
\left( z^{\ast },w^{\ast }\right) &\rightarrow &z^{\ast }\odot w^{\ast
}=\left( \overset{.}{a_{1}},\overset{..}{b_{1}}\right) \odot \left( \overset{%
.}{a_{2}},\overset{..}{b_{2}}\right) =\left( \alpha \left(
a_{1}a_{2}-b_{1}b_{2}\right) ,\beta \left( a_{1}b_{2}+b_{1}a_{2}\right)
\right) , \\
\overset{..}{\parallel }.\overset{..}{\parallel } &:&%
\mathbb{C}
\left( N\right) \rightarrow 
\mathbb{R}
_{\beta },z^{\ast }\rightarrow \overset{..}{\parallel }z^{\ast }\overset{..}{%
\parallel }=\sqrt[..]{\left[ \imath \left( \overset{.}{a}_{1}\overset{.}{-}%
\overset{.}{0}\right) \right] ^{\overset{..}{2}}\overset{..}{+}\left( 
\overset{..}{b}_{1}\overset{..}{-}\overset{..}{0}\right) ^{\overset{..}{2}}}%
=\beta \left( \sqrt{a_{1}^{2}+b_{1}^{2}}\right)
\end{eqnarray*}%
for all $z^{\ast }=\left( \overset{.}{a}_{1},\overset{..}{b}_{1}\right)
,w^{\ast }=\left( \overset{.}{a}_{2},\overset{..}{b}_{2}\right) \in 
\mathbb{C}
\left( N\right) $ where $0_{%
\mathbb{C}
\left( N\right) }=\left( \overset{.}{0},\overset{..}{0}\right) .$ The set $%
\mathbb{C}
\left( N\right) $ is also called $\ast -$complex field. In addition, $%
\overset{..}{\parallel }z^{\ast }\odot w^{\ast }\overset{..}{\parallel }=%
\overset{..}{\parallel }z^{\ast }\overset{..}{\parallel }\overset{..}{\times 
}\overset{..}{\parallel }w^{\ast }\overset{..}{\parallel }$ for all $z^{\ast
},w^{\ast }\in 
\mathbb{C}
\left( N\right) $ \cite{2}.

The $\ast -$complex conjugate $\overline{z^{\ast }}$ of a $\ast -$complex
number $z^{\ast }=\left( \overset{.}{a},\overset{..}{b}\right) \in 
\mathbb{C}
\left( N\right) $ is defined by%
\begin{equation*}
\overline{z^{\ast }}=\left( \alpha \left( a\right) ,\beta \left\{ -\beta
^{-1}\left( \overset{..}{b}\right) \right\} \right) =\left( \overset{.}{a},%
\overset{..}{-}\overset{..}{b}\right)
\end{equation*}%
and $\overset{..}{\parallel }z^{\ast }\odot \overline{z^{\ast }}\overset{..}{%
\parallel }=\overset{..}{\parallel }z^{\ast }\overset{..}{\parallel }^{%
\overset{..}{2}}$ \cite{5}.

For $z^{\ast }=\left( \overset{.}{a},\overset{..}{b}\right) \in 
\mathbb{C}
\left( N\right) $ and $w^{\ast }=\left( \overset{.}{c},\overset{..}{d}%
\right) \in 
\mathbb{C}
\left( N\right) -\left\{ 0_{%
\mathbb{C}
\left( N\right) }\right\} ,$ the $\ast -$division $z\oslash w$ is defined by%
\begin{equation*}
z^{\ast }\oslash w^{\ast }=\left( \overset{.}{a},\overset{..}{b}\right)
\oslash \left( \overset{.}{c},\overset{..}{d}\right) =\left( \alpha \left\{ 
\frac{ac+bd}{c^{2}+d^{2}}\right\} ,\beta \left\{ \frac{bc-ad}{c^{2}+d^{2}}%
\right\} \right)
\end{equation*}%
\cite{5}.

By simple calculations, it can be shown that $\overline{z^{\ast }\oslash
w^{\ast }}=\overline{z^{\ast }}\oslash \overline{w^{\ast }}$ for $z^{\ast
}=\left( \overset{.}{a},\overset{..}{b}\right) \in 
\mathbb{C}
\left( N\right) $ and $w^{\ast }=\left( \overset{.}{c},\overset{..}{d}%
\right) \in 
\mathbb{C}
\left( N\right) -\left\{ 0_{%
\mathbb{C}
\left( N\right) }\right\} $ which is used in Proposition 4 ii). Also, we can
denote non-Newtonian complex number $z^{\ast }=\left( \overset{.}{a},\overset%
{..}{b}\right) $ by $\left( \overset{.}{a},\overset{..}{0}\right) \oplus i_{%
\mathbb{C}
\left( N\right) }\odot \left( \overset{.}{0},\overset{..}{b}\right) $ where $%
i_{%
\mathbb{C}
\left( N\right) }=\left( \overset{.}{0},\overset{..}{1}\right) ,\left( i_{%
\mathbb{C}
\left( N\right) }\right) ^{2}=\ominus 1_{%
\mathbb{C}
\left( N\right) }=\ominus \left( \overset{.}{1},\overset{..}{0}\right) .$

\begin{definition}
A conjugate-linear map $x\rightarrow x^{\ast }$ of a complex algebra $%
\mathbb{A}$ into $\mathbb{A}$ is called an involution on $\mathbb{A}$ if the
following properties hold for all $x,y\in \mathbb{A}:$
\end{definition}

\textit{(ii) }$\left( x^{\ast }\right) ^{\ast }=x,$

\textit{(iii) }$\left( xy\right) ^{\ast }=y^{\ast }x^{\ast }.$

\textit{The pair }$\left( \mathbb{A},\ast \right) $\textit{\ is called an }$%
\ast -$\textit{algebra. A Banach }$\ast -$\textit{algebra is a }$\ast -$%
\textit{algebra }$\mathbb{A}$\textit{\ together with a complete
submultiplicative norm such that }$||x^{\ast }||=||x||$\textit{\ for }$x\in 
\mathbb{A}$\textit{. If, in addition, A has a unit such that }$||1||=1$%
\textit{, we call }$\mathbb{A}$\textit{\ a unital Banach }$\ast -$\textit{%
algebra. A }$C^{\ast }-$\textit{algebra is a Banach }$\ast -$\textit{algebra 
}$\mathbb{A}$\textit{\ such that for all }$x$\textit{\ in }$\mathbb{A}$%
\begin{equation*}
\left\Vert x^{\ast }x\right\Vert =\left\Vert x\right\Vert ^{2}
\end{equation*}%
\textit{\cite{9}.}

Motivated by these facts, our purpose in the current study is to introduce
the concept of a $C^{\ast }-$algebra with respect to the $\ast -$calculus
using $\ast -$complex field $%
\mathbb{C}
\left( N\right) $ instead of complex field $%
\mathbb{C}
$ in its classical definition and is to give some of its properties. Also,
we construct some nontrivial examples to validate our hypotheses and to
verify the usability of our findings. Comparing with works with respect to $%
C^{\ast }-$algebras in the literature, we believe that our results will
create new and different perspective in functional analysis and will serve
as an introduction and standard reference for the specialized articles in
the future works and applications on current research topics in the subject.

\section{Main Results}

In the literature we have observed that some topological concepts have not
been defined for $\ast -$complex field $%
\mathbb{C}
\left( N\right) $ in the sense of $\ast -$calculus. Therefore, before giving
our main findings, let us first introduce some of the concepts which will be
needed in the sequel.

\begin{definition}
$\delta $ $N-$neighborhood of a $\ast -$point $z_{0}\in 
\mathbb{C}
\left( N\right) $ is the set of all $\ast -$complex numbers $z$ such that $%
\overset{..}{\parallel }z\ominus z_{0}\overset{..}{\parallel }\overset{..}{<}%
\delta $ where $\delta $ is any given $\beta -$positive number. A deleted $%
\delta $ $N-$neighborhood of $z_{0}$ is omitted, that is, $\overset{..}{0}%
\overset{..}{<}\overset{..}{\parallel }z\ominus z_{0}\overset{..}{\parallel }%
\overset{..}{<}\delta .$
\end{definition}

\begin{definition}
Let $S\subset 
\mathbb{C}
\left( N\right) .$ A $\ast -$point $z_{0}\in 
\mathbb{C}
\left( N\right) $ is said to be an $N-$interior point of a set $S$ if there
exists a $\delta $ $N-$neighborhood of $z_{0}$ all of whose $\ast -$points
belong to $S.$ If every $\ast -$point $z$ of a set $S$ is an $N-$interior
point, then $S$ is said to be an $N-$open set.
\end{definition}

\begin{definition}
Let $A=\left\{ A_{i}:i\in I\right\} $ be family of $N-$open sets at $%
\mathbb{C}
\left( N\right) .$ If $E\subset 
\mathbb{C}
\left( N\right) $ and $E\subset \underset{i\in I}{\cup }A_{i},$ then the
family $A$ is called an $N-$open cover of the set $E.$ If $I_{0}\subset I$
is a finite subset and $E\subset \underset{i\in I_{0}}{\cup }A_{i},$ then
the family $A_{0}=\left\{ A_{i}:i\in I_{0}\right\} $ is called a finite $N-$%
subcover of the set $E.$
\end{definition}

\begin{definition}
A set $E\subset 
\mathbb{C}
\left( N\right) $ is said to be $\ast -$compact if every $N-$open cover of $%
A $ has a finite $N-$subcover.
\end{definition}

\begin{definition}
A $\ast -$complex function is a function $f$ whose domain and range are
subsets of the set $%
\mathbb{C}
\left( N\right) $ of $\ast -$complex numbers.
\end{definition}

\begin{definition}
Let $f$ be a $\ast -$continuous function defined in a deleted $N-$%
neighborhood of $z_{0}$ and $L\in 
\mathbb{C}
\left( N\right) .$ The $N-$limit of $f$ as $z$ tends to $z_{0}$ exists and
is equal to $L,$ written as $\underset{z\rightarrow z_{0}}{\lim }f\left(
z\right) =L,$ if for every $\varepsilon \overset{..}{>}\overset{..}{0},$
there exists a $\delta \overset{..}{>}\overset{..}{0},$ such that $\overset{%
..}{\parallel }f\left( z\right) \ominus L\overset{..}{\parallel }\overset{..}%
{<}\varepsilon $ whenever $\overset{..}{0}\overset{..}{<}\overset{..}{%
\parallel }z\ominus z_{0}\overset{..}{\parallel }\overset{..}{<}\delta .$
\end{definition}

\begin{definition}
A $\ast -$complex function $f$ is continuous at a point $z_{0}$ if $\underset%
{z\rightarrow z_{0}}{\lim }f\left( z\right) =f\left( z_{0}\right) .$
\end{definition}

In the later sections, we need the concept of a vector space constructed on $%
\ast -$complex field $%
\mathbb{C}
\left( N\right) $ to define the concepts of an algebra, a Banach algebra and
a $C^{\ast }-$algebra in the sense of $\ast -$calculus.

Kadak and Efe \cite{10} gave the definition of a vector space over the
non-Newtonian field as follows:

Let $\mathbb{F}^{\ast }$ denotes either the non-Newtonian real field $%
\mathbb{R}
\left( N\right) $ or the non-Newtonian complex field $%
\mathbb{C}
^{\ast }.$ A non-Newtonian vector space ($N-$vector space) over the field $%
\mathbb{F}^{\ast }$ is a set $V$ on which two operations are defined, called 
$\ast -$addition and scalar $\ast -$multiplication, denoted by $+$ and $%
\times $ by%
\begin{eqnarray*}
+ &:&V\times V\rightarrow V \\
\left( u,v\right) &\rightarrow &u+v=\left( \overset{.}{u}_{1}\overset{.}{+}%
\overset{.}{v}_{1},\overset{..}{u}_{2}\overset{..}{+}\overset{..}{v}%
_{2}\right) , \\
\times &:&\mathbb{F}^{\ast }\times V\rightarrow V \\
\left( \lambda ,u\right) &\rightarrow &\lambda \times u=\left( \overset{.}{%
\lambda }\overset{.}{\times }\overset{.}{u}_{1},\overset{..}{\lambda }%
\overset{..}{\times }\overset{..}{u}_{2}\right) ,
\end{eqnarray*}%
where the $\ast -$vectors are $u=\left( \overset{.}{u}_{1},\overset{..}{u}%
_{2}\right) ,$ $v=\left( \overset{.}{v}_{1},\overset{..}{v}_{2}\right) \in X$
and the scalar $\lambda =\left( \overset{.}{\lambda },\overset{..}{\lambda }%
\right) \in \mathbb{F}^{\ast }.$ Then the operations must satisfy the
following conditions.

(i) \textit{Closure}. For all $\lambda \in \mathbb{F}^{\ast }$ and all $%
u,v\in V,$ the $\ast -$sum $u+v$ and the scalar $\ast -$product $\lambda
\times v$ are uniquely defined and belong to $V.$

(ii) \textit{Associativity}. For all $\xi ,\eta \in \mathbb{F}^{\ast }$ and
all $u,v,w\in V$ then $u+\left( v+w\right) =\left( u+v\right) +w$ and $\xi
\times \left( \eta \times v\right) =\left( \xi \times \eta \right) \times v.$

(iii) \textit{Additive Commutativity}. For all $u,v\in V$ then $u+v=v+u.$

(iv) \textit{Additive Identity}. The set $V$ contains an additive $\ast -$%
identity element, denoted by $\theta _{A}=\left( \overset{.}{0},\overset{..}{%
0}\right) ,$ such that for all $u\in V,$ $u+\theta _{A}=u.$

(v) \textit{Additive Inverse}. The set $V$ contains an additive $\ast -$%
inverse element, denoted by $u_{+}^{-1}=\left( \overset{.}{-}\overset{.}{u}%
_{1},\overset{..}{-}\overset{..}{u}_{2}\right) \in V,$ such that for all $%
u\in V,$ $u+u_{+}^{-1}=u_{+}^{-1}+u=\theta _{A}.$

(vi) \textit{Multiplicative Identity}. The set $V$ contains a multiplicative 
$\ast -$identity element, denoted by $\theta _{M}=\left( \overset{.}{1},%
\overset{..}{0}\right) ,$ such that for all $u\in V,$ $u\times \theta
_{M}=u. $

(vii) \textit{Multiplicative Inverse}. The set $V$ contains a multiplicative 
$\ast -$inverse element, denoted by $u_{\times }^{-1}=\left( \overset{.}{m},%
\overset{..}{n}\right) \in V,$ where $\overset{.}{m}=\alpha \left\{
u_{1}/\left( u_{1}^{2}+u_{2}^{2}\right) \right\} $ and $\overset{..}{n}%
=\beta \left\{ -u_{2}/\left( u_{1}^{2}+u_{2}^{2}\right) \right\} $ such that
for all $u\in V,$ $u\times u_{\times }^{-1}=u_{\times }^{-1}\times u=\theta
_{M}.$

(viii) \textit{Distributive Laws}. For all $\xi ,\eta \in \mathbb{F}^{\ast }$
and all $u,v\in V$ then $\xi \times \left( u+v\right) =\left( \xi \times
u\right) +\left( \xi \times v\right) $ and $\left( \xi +\eta \right) \times
u=\left( \xi \times u\right) +\left( \eta \times u\right) .$

In this definition, it has been noticed that the $\ast -$points $\left( 
\overset{.}{a},\overset{..}{b}\right) $ of the $\ast -$complex field $%
\mathbb{C}
\left( N\right) $ with $a=b$ are considered for the definition of a vector
space over $%
\mathbb{C}
\left( N\right) $, that is, the concept of a vector space is constructed on
the subset of $\ast -$complex field $%
\mathbb{C}
\left( N\right) $ consisting of points $\left( \overset{.}{a},\overset{..}{b}%
\right) $ where $a=b$. So, we need to redefine the concept as follows:

\begin{definition}
A non-Newtonian vector space over the field $%
\mathbb{C}
\left( N\right) $ $\left( 
\mathbb{C}
\left( N\right) -\text{vector space}\right) $ is a set $\mathbb{A}$ together
with the following operations%
\begin{eqnarray*}
\widehat{+} &:&\mathbb{A\times A\rightarrow A}\text{, }\left( x,y\right)
\rightarrow x\widehat{+}y \\
\widehat{\cdot } &:&%
\mathbb{C}
\left( N\right) \times \mathbb{A\rightarrow A}\text{, }\left( \lambda
,x\right) \rightarrow \lambda \widehat{\cdot }x
\end{eqnarray*}%
called $N-$addition and $N-$scalar multiplicaiton, respectively, such that
\end{definition}

\qquad \textit{(i) }$x\widehat{+}y=y\widehat{+}x$\textit{\ for all }$x,y\in
A $\textit{,}

\textit{\qquad (ii) }$\left( x\widehat{+}y\right) \widehat{+}z=x\widehat{+}%
\left( y\widehat{+}z\right) $\textit{\ for all }$x,y,z\in A$\textit{,}

\textit{\qquad (iii) there exists a vector }$0_{\mathbb{A}}$\textit{\ in }$A$%
\textit{\ such that }$x\widehat{+}0_{\mathbb{A}}=x$\textit{\ for all }$x\in
A $\textit{,}

\textit{\qquad (iv) for each }$x\in A$\textit{, there exists a vector }$%
\widehat{-}x$\textit{\ in }$A$\textit{\ such that }$x\widehat{+}\left( 
\widehat{-}x\right) =0_{\mathbb{A}}$\textit{,}

\textit{\qquad (v) }$\lambda \widehat{\cdot }\left( x\widehat{+}y\right)
=\lambda \widehat{\cdot }x\widehat{+}\lambda \widehat{\cdot }y$\textit{\ for
all }$x,y\in A$\textit{\ and }$\lambda \in 
\mathbb{C}
\left( N\right) $\textit{,}

\textit{\qquad (vi) }$\left( \lambda \oplus \mu \right) \widehat{\cdot }%
x=\lambda \widehat{\cdot }x\widehat{+}\mu \widehat{\cdot }y$\textit{\ for
all }$x\in A$\textit{\ and }$\lambda ,\mu \in 
\mathbb{C}
\left( N\right) $\textit{,}

\textit{\qquad (vii) }$\left( \lambda \odot \mu \right) \widehat{\cdot }%
x=\lambda \widehat{\cdot }\left( \mu \widehat{\cdot }x\right) $\textit{\ for
all }$x\in A$\textit{\ and }$\lambda ,\mu \in 
\mathbb{C}
\left( N\right) $\textit{,}

\textit{\qquad (viii) }$1_{%
\mathbb{C}
\left( N\right) }\widehat{\cdot }x=x$\textit{\ for all }$x\in A$\textit{.}

Similarly, it can be simply defined the concept of a normed space over $\ast
-$complex field $%
\mathbb{C}
\left( N\right) $ using $%
\mathbb{R}
_{\beta }-$valued norm $\overset{..}{\parallel }.\overset{..}{\parallel }$
instead of real valued norm $\left\Vert .\right\Vert $ in its classical
definition, as follows:

\begin{definition}
Let $\mathbb{A}$ be a $%
\mathbb{C}
\left( N\right) -$vector space. A norm on $\mathbb{A}$ is a function $%
\left\Vert .\right\Vert _{\mathbb{A}}:\mathbb{A\rightarrow 
\mathbb{R}
}_{\beta }$ such that the following axioms are satisfied for all $x,y\in 
\mathbb{A}$ and $\lambda \in 
\mathbb{C}
\left( N\right) :$
\end{definition}

\textit{(i) }$\left\Vert x\right\Vert _{\mathbb{A}}=\overset{..}{0}$\textit{%
\ implies }$x=0_{\mathbb{A}}.$

\textit{(ii) }$\left\Vert \lambda \widehat{\cdot }x\right\Vert _{\mathbb{A}}=%
\overset{..}{\parallel }\lambda \overset{..}{\parallel }\overset{..}{\times }%
\left\Vert x\right\Vert _{\mathbb{A}}.$

\textit{(iii) }$\left\Vert x\widehat{+}y\right\Vert _{\mathbb{A}}\overset{..}%
{\leq }\left\Vert x\right\Vert _{\mathbb{A}}\overset{..}{+}\left\Vert
y\right\Vert _{\mathbb{A}}.$

\textit{Also, }$A$\textit{\ is called a }$%
\mathbb{C}
\left( N\right) -$\textit{normed space.}

\begin{remark}
Each $%
\mathbb{C}
\left( N\right) -$normed space $\left( \mathbb{A}\text{,}\left\Vert
.\right\Vert _{\mathbb{A}}\right) $ is a non-Newtonian metric space $\left( 
\mathbb{A}\text{,}d\right) $ with a non-Newtonian metric given by $d\left(
x,y\right) =\left\Vert x\widehat{-}y\right\Vert _{\mathbb{A}}.$ All concepts
from non-Newtonian metric spaces are applicable to $%
\mathbb{C}
\left( N\right) -$normed spaces.
\end{remark}

We will use the following new concepts and properties in next section.

\begin{definition}
Let $\left( \mathbb{A}\text{,}\left\Vert .\right\Vert _{\mathbb{A}}\right) $ 
\textit{be a }$%
\mathbb{C}
\left( N\right) -$normed space\textit{, }$\left( x_{n}\right) $ be any
sequence in $\mathbb{A}$ and $x\in \mathbb{A}.$ If for every $\varepsilon 
\overset{..}{>}\overset{..}{0}$ there exists $n_{0}\in 
\mathbb{N}
$ depending on $\varepsilon $ such that $\left\Vert x_{n}\widehat{-}%
x\right\Vert _{\mathbb{A}}\overset{..}{<}\varepsilon $ for all $n\geq n_{0}$
then we say that $\left( x_{n}\right) $ converges to $x$ with respect to the
norm $\left\Vert .\right\Vert _{\mathbb{A}}.$
\end{definition}

\qquad \textit{If for every }$\varepsilon \overset{..}{>}\overset{..}{0}$%
\textit{\ there exists }$n_{0}\in 
\mathbb{N}
$\textit{\ depending on }$\varepsilon $\textit{\ such that\ }$\left\Vert
x_{n}\widehat{-}x_{m}\right\Vert _{\mathbb{A}}\overset{..}{<}\varepsilon $%
\textit{\ for all }$n,m\geq n_{0}$ \textit{then we say that }$\left(
x_{n}\right) $\textit{\ is a Cauchy sequence with respect to the norm }$%
\left\Vert .\right\Vert _{\mathbb{A}}.$\textit{\ }

\begin{definition}
Let $\mathbb{A}$ be a $%
\mathbb{C}
\left( N\right) -$normed space. If every Cauchy sequence $\left(
x_{n}\right) \subset \mathbb{A}$ converges to the limit in $\mathbb{A}$ 
\textit{with respect to the norm }$\left\Vert .\right\Vert _{\mathbb{A}}$ we
say that $\mathbb{A}$ is complete \textit{with respect to the norm }$%
\left\Vert .\right\Vert _{\mathbb{A}}$. A complete $%
\mathbb{C}
\left( N\right) -$normed space is called $%
\mathbb{C}
\left( N\right) -$Banach space.
\end{definition}

\begin{definition}
Let $\mathbb{A}$ be a $%
\mathbb{C}
\left( N\right) -$normed space and $\left( x_{n}\right) $ be any sequence in 
$\mathbb{A}$. An infinite sum%
\begin{equation*}
x_{1}\widehat{+}x_{2}\widehat{+}...\widehat{+}x_{n}\widehat{+}...=\underset{%
n=1}{\overset{\infty }{\widehat{\sum }}}x_{n}
\end{equation*}%
is called the series in $\mathbb{A}$. A sequence $\left( S_{m}\right) $ with
the general term $S_{m}=\underset{n=1}{\overset{m}{\widehat{\sum }}}x_{n}$
is called partial sums sequence of the series $\underset{n=1}{\overset{%
\infty }{\widehat{\sum }}}x_{n}.$ If the sequence $\left( S_{m}\right) $
converges to some $S\in \mathbb{A}$ with respect to the norm $\left\Vert
.\right\Vert _{\mathbb{A}}$, then it is said that the series $\underset{n=1}{%
\overset{\infty }{\widehat{\sum }}}x_{n}$ converges to $S\in \mathbb{A}$
with respect to the norm $\left\Vert .\right\Vert _{\mathbb{A}}$.
\end{definition}

\begin{definition}
Let $\mathbb{A}$ be a $%
\mathbb{C}
\left( N\right) -$normed space and $\underset{n=1}{\overset{\infty }{%
\widehat{\sum }}}x_{n}$ be any series in $\mathbb{A}$. If the $\beta -$%
series $_{\beta }\underset{n=1}{\overset{\infty }{\sum }}\left\Vert
x_{n}\right\Vert _{\mathbb{A}}$ is $\beta -$convergent, we say that $%
\underset{n=1}{\overset{\infty }{\widehat{\sum }}}x_{n}$ is absolutely
convergent in $\mathbb{A}$ with respect to the norm $\left\Vert .\right\Vert
_{\mathbb{A}}$.
\end{definition}

\begin{theorem}
A $%
\mathbb{C}
\left( N\right) -$normed space $\mathbb{A}$ is a $%
\mathbb{C}
\left( N\right) -$Banach space if and only if every absolutely convergent
series in $\mathbb{A}$ is convergent in $\mathbb{A}$ with respect to the
norm $\left\Vert .\right\Vert _{\mathbb{A}}$.
\end{theorem}

\begin{proof}
Let $\mathbb{A}$ be a $%
\mathbb{C}
\left( N\right) -$Banach space and $\underset{n=1}{\overset{\infty }{%
\widehat{\sum }}}x_{n}$ be absolutely convergent in $\mathbb{A}$ with
respect to the norm $\left\Vert .\right\Vert _{\mathbb{A}}$. Given $%
\varepsilon \overset{..}{>}\overset{..}{0},$ choose $n_{0}\in 
\mathbb{N}
$ such that $_{\beta }\underset{n=n_{0}}{\overset{\infty }{\sum }}\left\Vert
x_{n}\right\Vert _{\mathbb{A}}\overset{..}{<}\varepsilon .$ This implies that%
\begin{equation*}
\left\Vert S_{n}\widehat{-}S_{m}\right\Vert _{\mathbb{A}}=\left\Vert 
\underset{k=n+1}{\overset{m}{\widehat{\sum }}}x_{k}\right\Vert _{\mathbb{A}}%
\overset{..}{\leq }_{\beta }\underset{k=n+1}{\overset{m}{\sum }}\left\Vert
x_{k}\right\Vert _{\mathbb{A}}\overset{..}{<}\varepsilon
\end{equation*}%
for all $n,m\geq n_{0}$ where $n<m.$ Thus, $\left( S_{n}\right) $ is a
Cauchy sequence with respect to the norm $\left\Vert .\right\Vert _{\mathbb{A%
}}$ and so it converges to some $S\in \mathbb{A}$ with respect to the norm $%
\left\Vert .\right\Vert _{\mathbb{A}}.$ Then, the series $\underset{n=1}{%
\overset{\infty }{\widehat{\sum }}}x_{n}$ converges to $S\in \mathbb{A}$
with respect to the norm $\left\Vert .\right\Vert _{\mathbb{A}}.$

Conversely, suppose that every absolutely convergent series in $\mathbb{A}$
is convergent in $\mathbb{A}$ with respect to the norm $\left\Vert
.\right\Vert _{\mathbb{A}}.$ Let $\left( x_{n}\right) $ be any Cauchy
sequence with respect to the norm $\left\Vert .\right\Vert _{\mathbb{A}}.$
Then, we can choose $n_{k}\in 
\mathbb{N}
$ such that%
\begin{equation*}
n_{k}<n_{k+1},\text{ \ }\left\Vert x_{n_{k}}\widehat{-}x_{n_{k+1}}\right%
\Vert _{\mathbb{A}}\overset{..}{<}\frac{\overset{..}{1}}{\overset{..}{2}%
^{\left( k\right) _{\beta }}}\beta
\end{equation*}%
for any $k\in 
\mathbb{N}
.$ This implies that the series $x_{n_{1}}\widehat{+}\underset{k=1}{\overset{%
\infty }{\widehat{\sum }}}\left( x_{n_{k+1}}\widehat{-}x_{n_{k}}\right) $ is
absolutely convergent with respect to the norm $\left\Vert .\right\Vert _{%
\mathbb{A}}$ since $_{\beta }\underset{k=1}{\overset{\infty }{\sum }}\frac{%
\overset{..}{1}}{\overset{..}{2}^{\left( k\right) _{\beta }}}\beta $ is $%
\beta -$convergent by Example 1 in \cite{7}, and so it converges to some $%
x\in \mathbb{A}$ with respect to the norm $\left\Vert .\right\Vert _{\mathbb{%
A}}.$ Therefore, we write $x_{n_{k}}=x_{n_{1}}\widehat{+}\underset{j=1}{%
\overset{k-1}{\widehat{\sum }}}\left( x_{n_{j+1}}\widehat{-}x_{n_{j}}\right) 
$ and $\left( x_{n_{k}}\right) $ converges to $x\in \mathbb{A}$ with respect
to the norm $\left\Vert .\right\Vert _{\mathbb{A}}$ as $k\rightarrow \infty
. $ On the other hand, since $\left( x_{n}\right) $ is a Cauchy sequence
with respect to the norm $\left\Vert .\right\Vert _{\mathbb{A}}$ and $\left(
x_{n_{k}}\right) $ converges to $x\in \mathbb{A}$ with respect to the norm $%
\left\Vert .\right\Vert _{\mathbb{A}}$ as $k\rightarrow \infty ,$ for every $%
\varepsilon \overset{..}{>}\overset{..}{0}$\ there exists $N,M\in 
\mathbb{N}
$\ depending on $\varepsilon $\ such that%
\begin{equation*}
\left\Vert x_{n_{k}}\widehat{-}x\right\Vert _{\mathbb{A}}\overset{..}{<}%
\frac{\varepsilon }{\overset{..}{2}}\beta \ \text{\ for all }k\geq N
\end{equation*}%
and%
\begin{equation*}
\left\Vert x_{n}\widehat{-}x_{m}\right\Vert _{\mathbb{A}}\overset{..}{<}%
\frac{\varepsilon }{\overset{..}{2}}\beta \ \text{\ for all }n,m\geq M.
\end{equation*}%
Taking $n_{k}$ such that $k\geq N$ and $n_{k}\geq M,$ then%
\begin{equation*}
\left\Vert x_{n}\widehat{-}x\right\Vert _{\mathbb{A}}\overset{..}{\leq }%
\left\Vert x_{n}\widehat{-}x_{n_{k}}\right\Vert _{\mathbb{A}}\overset{..}{+}%
\left\Vert x_{n_{k}}\widehat{-}x\right\Vert _{\mathbb{A}}\overset{..}{<}%
\varepsilon
\end{equation*}%
for all $n\geq M.$ So, $\left( x_{n}\right) $ is convergent in $A$ with
respect to the norm $\left\Vert .\right\Vert _{\mathbb{A}}.$ The proof is
completed.
\end{proof}

\subsection{Non-Newtonian Banach Algebras}

This section is fundamentals of Non-Newtonian Banach Algebras.

We begin this section by introducing the concept of a Banach $%
\mathbb{C}
\left( N\right) -$algebra.

\begin{definition}
A non-Newtonian complex algebra ($%
\mathbb{C}
\left( N\right) -$algebra) is a vector space $\mathbb{A}$ over $%
\mathbb{C}
\left( N\right) $ in which an associative and distributive $N-$%
multiplication is defined, that is%
\begin{eqnarray*}
x\widehat{\times }\left( y\widehat{\times }z\right) &=&\left( x\widehat{%
\times }y\right) \widehat{\times }z, \\
\left( x\widehat{+}y\right) \widehat{\times }z &=&x\widehat{\times }z%
\widehat{+}y\widehat{\times }z,\text{ }x\widehat{\times }\left( y\widehat{+}%
z\right) =x\widehat{\times }y\widehat{+}x\widehat{\times }z
\end{eqnarray*}%
for all $x,y,z\in \mathbb{A}$ and which is related to scalar $N-$%
multiplication so that%
\begin{equation*}
\lambda \widehat{\cdot }\left( x\widehat{\times }y\right) =x\widehat{\times }%
\left( \lambda \widehat{\cdot }y\right) =\left( \lambda \widehat{\cdot }%
x\right) \widehat{\times }y
\end{equation*}%
\textit{for all }$x,y\in \mathbb{A}$\textit{\ and for all }$\lambda \in 
\mathbb{C}
\left( N\right) .$ An $%
\mathbb{C}
\left( N\right) -$algebra is unital if it has a unit element $1_{\mathbb{A}%
}^{N}$ with $x\widehat{\times }1_{\mathbb{A}}^{N}=1_{\mathbb{A}}^{N}\widehat{%
\times }x=x$\textit{\ for all }$x\in \mathbb{A}$.
\end{definition}

\textit{\qquad If there is a norm defined in }$\mathbb{A}$\textit{\ which
makes }$\mathbb{A}$\textit{\ into a normed space and which satisfies the }$%
N- $\textit{multiplicative inequality}%
\begin{equation*}
\overset{..}{\parallel }x\widehat{\times }y\overset{..}{\parallel }_{\mathbb{%
A}}\overset{..}{\leq }\overset{..}{\parallel }x\overset{..}{\parallel }_{%
\mathbb{A}}\overset{..}{\times }\overset{..}{\parallel }y\overset{..}{%
\parallel }_{\mathbb{A}}
\end{equation*}%
\textit{for all }$x,y\in \mathbb{A}$\textit{, then }$\mathbb{A}$\textit{\ is
called a normed }$%
\mathbb{C}
\left( N\right) -$\textit{algebra. If, in addition, }$\mathbb{A}$\textit{\
is a Banach space with the norm }$\overset{..}{\parallel }.\overset{..}{%
\parallel }_{\mathbb{A}},$\textit{\ then we call }$\mathbb{A}$\textit{\ a
Banach }$%
\mathbb{C}
\left( N\right) -$\textit{algebra (non-Newtonian Banach algebra). }

\begin{remark}
\textit{The zero element in }$\mathbb{A}$\textit{\ is the element }$0_{%
\mathbb{A}}^{N}$\textit{\ such that }$x\widehat{+}0_{\mathbb{A}}^{N}=0_{%
\mathbb{A}}^{N}\widehat{+}x=x$\textit{\ for all }$x\in \mathbb{A}$\textit{.}
\end{remark}

The followings are examples of Banach $%
\mathbb{C}
\left( N\right) -$algebras.

\begin{example}
The Banach space $%
\mathbb{C}
\left( N\right) $ forms a Banach $%
\mathbb{C}
\left( N\right) -$algebra with the algebraic operations $\oplus $ and $\odot 
$ and the norm $\overset{..}{\parallel }.\overset{..}{\parallel }$ defined
on it.
\end{example}

\begin{example}
Let $\Omega \subset 
\mathbb{C}
\left( N\right) $ be compact and $C_{N}\left( \Omega \right) =\left\{
f:\Omega \rightarrow 
\mathbb{C}
\left( N\right) :f\text{\ is }\ast -\text{continouous}\right\} .$ Then, the
set $C_{N}\left( \Omega \right) $ forms a Banach space with respect to the
algebraic operations addition $\left( +\right) $ and scalar multiplication $%
\left( \cdot \right) $ and the norm on $C_{N}\left( \Omega \right) $ defined
as%
\begin{eqnarray*}
+ &:&C_{N}\left( \Omega \right) \times C_{N}\left( \Omega \right)
\rightarrow C_{N}\left( \Omega \right) ,\left( f,g\right) \rightarrow f+g \\
f+g &:&\Omega \rightarrow 
\mathbb{C}
\left( N\right) ,z\rightarrow \left( f+g\right) \left( z\right) =f\left(
z\right) \overset{..}{+}g\left( z\right) \\
\cdot &:&%
\mathbb{C}
\left( N\right) \times C_{N}\left( \Omega \right) \rightarrow C_{N}\left(
\Omega \right) ,\left( \lambda ,f\right) \rightarrow \lambda \cdot f \\
\lambda \cdot f &:&\Omega \rightarrow 
\mathbb{C}
\left( N\right) ,z\rightarrow \left( \lambda \cdot f\right) \left( z\right)
=\lambda \overset{..}{\times }f\left( z\right) \\
\left\Vert .\right\Vert _{C_{N}\left( \Omega \right) } &:&C_{N}\left( \Omega
\right) \rightarrow 
\mathbb{R}
_{\beta },f\rightarrow \left\Vert f\right\Vert _{C_{N}\left( \Omega \right)
}=\underset{z\in \Omega }{\max }\overset{..}{\parallel }f\left( z\right) 
\overset{..}{\parallel }
\end{eqnarray*}%
for all $f,g\in C_{N}\left( \Omega \right) $ and for all $\lambda \in 
\mathbb{C}
\left( N\right) $ \textit{\cite{4}.}
\end{example}

\qquad \textit{As a new algebraic operation, we define multiplication }$%
\left( \times \right) $\textit{\ as follows:}%
\begin{eqnarray*}
\times &:&C_{N}\left( \Omega \right) \times C_{N}\left( \Omega \right)
\rightarrow C_{N}\left( \Omega \right) ,\left( f,g\right) \rightarrow
f\times g \\
f\times g &:&\Omega \rightarrow 
\mathbb{C}
\left( N\right) ,z\rightarrow \left( f\times g\right) \left( z\right)
=f\left( z\right) \overset{..}{\times }g\left( z\right)
\end{eqnarray*}%
\textit{\qquad \qquad \qquad \qquad for all }$f,g\in C_{N}\left( \Omega
\right) .$\textit{\ It can be easily seen that}%
\begin{eqnarray*}
f\times \left( g\times h\right) &=&\left( f\times g\right) \times h, \\
\left( f+g\right) \times h &=&f\times h+g\times h,\text{ }f\times \left(
g+h\right) =f\times g+f\times h, \\
\lambda \cdot \left( f\times g\right) &=&f\times \left( \lambda \cdot
g\right) =\left( \lambda \cdot f\right) \times g
\end{eqnarray*}%
\textit{for all }$f,g,h\in C_{N}\left( \Omega \right) $\textit{\ and for all 
}$\lambda \in 
\mathbb{C}
\left( N\right) .$\textit{\ Also, we have}%
\begin{eqnarray*}
\left\Vert f\times g\right\Vert _{C_{N}\left( \Omega \right) } &=&\underset{%
z\in \Omega }{\max }\overset{..}{\parallel }\left( f\times g\right) \left(
z\right) \overset{..}{\parallel }=\underset{z\in \Omega }{\max }\overset{..}{%
\parallel }f\left( z\right) \overset{..}{\times }g\left( z\right) \overset{..%
}{\parallel } \\
&=&\underset{z\in \Omega }{\max }\left( \overset{..}{\parallel }f\left(
z\right) \overset{..}{\parallel }\overset{..}{\times }\overset{..}{\parallel 
}g\left( z\right) \overset{..}{\parallel }\right) \\
&=&\underset{z\in \Omega }{\max }\overset{..}{\parallel }f\left( z\right) 
\overset{..}{\parallel }\overset{..}{\times }\underset{z\in \Omega }{\max }%
\overset{..}{\parallel }g\left( z\right) \overset{..}{\parallel } \\
&=&\left\Vert f\right\Vert _{C_{N}\left( \Omega \right) }\overset{..}{\times 
}\left\Vert g\right\Vert _{C_{N}\left( \Omega \right) }.
\end{eqnarray*}%
So, $C_{N}\left( \Omega \right) $ is \textit{a Banach }$%
\mathbb{C}
\left( N\right) -$\textit{algebra.}

\qquad Note that the zero and unit elements in $C_{N}\left( \Omega \right) $
are $\ast -$continuous functions $0_{%
\mathbb{C}
\left( N\right) }$ and $1_{%
\mathbb{C}
\left( N\right) },$ respectively, where $0_{%
\mathbb{C}
\left( N\right) }\left( z\right) =0_{%
\mathbb{C}
\left( N\right) }$ and $1_{%
\mathbb{C}
\left( N\right) }\left( z\right) =1_{%
\mathbb{C}
\left( N\right) }$ for all $z\in 
\mathbb{C}
\left( N\right) .$

Now, we define the concept of an $N-$subalgebra, and also, we investigate
some $N-$subalgebras of $C_{N}\left( \Omega \right) .$

\begin{definition}
Let $\mathbb{A}$ be a vector space over $%
\mathbb{C}
\left( N\right) $ and $\mathbb{B}$ be a subset of $\mathbb{A}$. If the
following axioms hold, then $\mathbb{B}$ is called an $N-$subspace of $%
\mathbb{A}$.

(i) The zero element $0_{\mathbb{A}}^{N}$ \ is in $\mathbb{B}$.

(ii) If $x$ and $y$ are elements of $\mathbb{B}$, then the sum $x\widehat{+}%
y $ is an element of $\mathbb{B}$.

(iii) If $x$ is an element of $\mathbb{B}$ and $\lambda $ is a scalar from $%
\mathbb{C}
\left( N\right) ,$ then the scalar product $\lambda \widehat{\cdot }x$ is an
element of $\mathbb{B}$.
\end{definition}

\begin{definition}
Let $\mathbb{A}$ be a $%
\mathbb{C}
\left( N\right) -$\textit{algebra and }$\mathbb{B}$ be an $N-$subspace of $%
\mathbb{A}$. If $x\widehat{\times }y\in \mathbb{B}$ for all $x,y\in \mathbb{B%
}$, then $\mathbb{B}$ is called an $N-$subalgebra of $\mathbb{A}$.
\end{definition}

\begin{example}
Consider the \textit{Banach }$%
\mathbb{C}
\left( N\right) -$\textit{algebra }$C_{N}\left( \Omega \right) $ given in
Example 2 for $\Omega =\left\{ z\in 
\mathbb{C}
\left( N\right) :\overset{..}{\parallel }z\overset{..}{\parallel }\overset{..%
}{\leq }\frac{\overset{..}{1}}{\overset{..}{2}}\beta \right\} .$ Define the
set%
\begin{equation*}
I_{z}=\left\{ f\in C_{N}\left( \Omega \right) :f\left( z\right) =0_{%
\mathbb{C}
\left( N\right) }\right\}
\end{equation*}%
for $z\in 
\mathbb{C}
\left( N\right) .$ Since $0_{C_{N}\left( \Omega \right) }^{N}=0_{%
\mathbb{C}
\left( N\right) }\in I_{z}$ and $f+g,$ $\lambda \cdot f,$ $f\times g\in
I_{z} $ for all $f,g\in I_{z}$ and for all $\lambda \in 
\mathbb{C}
\left( N\right) ,$ the set $I_{z}$ is an $N-$subalgebra of $C_{N}\left(
\Omega \right) .$
\end{example}

\begin{definition}
An $N-$polynomial over $%
\mathbb{C}
\left( N\right) $ in indeterminate $z$ is an expression of the form $f\left(
z\right) =z_{0}\oplus \left( z_{1}\odot z\right) \oplus \left( z_{2}\odot
z^{2}\right) \oplus ...\oplus \left( z_{n}\odot z^{n}\right) $ with $n\in 
\mathbb{N}
$ and $z_{0},z_{1},...,z_{n}\in 
\mathbb{C}
\left( N\right) .$ The set of such $N-$polynomials is denoted by $%
\mathbb{C}
\left( N\right) \left[ z\right] .$
\end{definition}

\begin{example}
Consider the \textit{Banach }$%
\mathbb{C}
\left( N\right) -$\textit{algebra }$C_{N}\left( \Omega \right) $ given in
Example 2 for $\Omega =\left\{ z\in 
\mathbb{C}
\left( N\right) :\overset{..}{\parallel }z\overset{..}{\parallel }\overset{..%
}{\leq }\frac{\overset{..}{1}}{\overset{..}{2}}\beta \right\} $. Obviously,
since $f$ is a $\ast -$continuous function for all $f\in 
\mathbb{C}
\left( N\right) \left[ z\right] ,$ $0_{C_{N}\left( \Omega \right) }^{N}=0_{%
\mathbb{C}
\left( N\right) }\in 
\mathbb{C}
\left( N\right) \left[ z\right] $ where the $0_{%
\mathbb{C}
\left( N\right) }$ function in $%
\mathbb{C}
\left( N\right) \left[ z\right] $ is obtained in the special case $%
z_{0}=z_{1}=...=z_{n}=0_{%
\mathbb{C}
\left( N\right) }$ and $f+g,$ $\lambda \cdot f,$ $f\times g\in 
\mathbb{C}
\left( N\right) \left[ z\right] $ for all $f,g\in 
\mathbb{C}
\left( N\right) \left[ z\right] $ and for all $\lambda \in 
\mathbb{C}
\left( N\right) ,$ the set $%
\mathbb{C}
\left( N\right) \left[ z\right] $ is an $N-$subalgebra of $C_{N}\left(
\Omega \right) .$
\end{example}

We continue this section with the concept of an $N-$invertible element and
some related results.

\begin{definition}
Let $\mathbb{A}$ be a unital Banach $%
\mathbb{C}
\left( N\right) -$algebra with a unit $1_{\mathbb{A}}^{N}.$ An element $x$
of $\mathbb{A}$ is called an $N-$invertible element, if there is an element $%
y$ in $\mathbb{A}$ such that $x\widehat{\times }y=y\widehat{\times }x=1_{%
\mathbb{A}}^{N}.$ The element $y$ is called an $N-$inverse of $x$ and
denoted by $x^{-1_{N}}.$ The $N-$inverse of an $N-$inbertible element is
unique. The set of $N-$invertible elements in $\mathbb{A}$ is denoted by $%
\mathbb{A}^{-1_{N}}.$ $\mathbb{A}^{-1_{N}}$ is a group with respect to the
multiplication $\widehat{\times }.$
\end{definition}

\begin{example}
Consider the \textit{Banach }$%
\mathbb{C}
\left( N\right) -$\textit{algebra }$C_{N}\left( \Omega \right) $ given in
Example 2 and the $N-$subalgebra $%
\mathbb{C}
\left( N\right) \left[ z\right] $ given in Example 4. Define the function $%
f:\Omega \rightarrow 
\mathbb{C}
\left( N\right) ,$ $f\left( z\right) =z\oplus 1_{%
\mathbb{C}
\left( N\right) }.$It is clear that $f\in 
\mathbb{C}
\left( N\right) \left[ z\right] .$ If we choose the $\ast -$continuous
function $g:\Omega \rightarrow 
\mathbb{C}
\left( N\right) ,$ $g\left( z\right) =\frac{1_{%
\mathbb{C}
\left( N\right) }}{z\oplus 1_{%
\mathbb{C}
\left( N\right) }},$ we get $f\left( z\right) \odot g\left( z\right)
=g\left( z\right) \odot f\left( z\right) =1_{%
\mathbb{C}
\left( N\right) }.$ This shows that the function $f$ is $N-$invertible in $%
\mathbb{C}
\left( N\right) \left[ z\right] $ and the function $g$ is $N-$inverse of $f.$
\end{example}

\begin{remark}
Let $\mathbb{A}$ be a unital Banach $%
\mathbb{C}
\left( N\right) -$algebra with a unit $1_{\mathbb{A}}^{N}$ and $x,y\in 
\mathbb{A}^{-1_{N}}.$ Then, $\left( x\widehat{\times }y\right)
^{-1_{N}}=y^{-1_{N}}\widehat{\times }x^{-1_{N}}$ and $\left(
x^{-1_{N}}\right) ^{-1_{N}}=x.$
\end{remark}

\begin{proposition}
Let $\mathbb{A}$ be a unital Banach $%
\mathbb{C}
\left( N\right) -$algebra with a unit $1_{\mathbb{A}}^{N}.$ If an element $%
x\in \mathbb{A}$ is in open ball $\overset{..}{\parallel }x\widehat{-}1_{%
\mathbb{A}}^{N}\overset{..}{\parallel }_{\mathbb{A}}\overset{..}{<}\overset{%
..}{1},$ then $x$ is $N-$invertible and the $N-$inverse $x^{-1_{N}}$ is
written as%
\begin{equation*}
x^{-1_{N}}=\underset{n=0}{\overset{\infty }{\sum }}\left( 1_{\mathbb{A}}^{N}%
\widehat{-}x\right) ^{n}.
\end{equation*}
\end{proposition}

\begin{proof}
Since $\underset{n=0}{\overset{\infty }{\overset{..}{\sum }}}\overset{..}{%
\parallel }1_{\mathbb{A}}^{N}\widehat{-}x\overset{..}{\parallel }_{\mathbb{A}%
}^{n}$ is $\beta -$convergent by Example 4 in \cite{7}, $\underset{n=0}{%
\overset{\infty }{\sum }}\left( 1_{\mathbb{A}}^{N}\widehat{-}x\right) ^{n}$
converges in norm $\overset{..}{\parallel }.\overset{..}{\parallel }_{%
\mathbb{A}}$ by Theorem 1. Let $y=\underset{n=0}{\overset{\infty }{\sum }}%
\left( 1_{\mathbb{A}}^{N}\widehat{-}x\right) ^{n}.$ Then, we get%
\begin{eqnarray*}
y\widehat{\times }x &=&x\widehat{\times }y=\left( 1_{\mathbb{A}}^{N}\widehat{%
-}\left( 1_{\mathbb{A}}^{N}\widehat{-}x\right) \right) \widehat{\times }y \\
&=&y\widehat{-}\left( 1_{\mathbb{A}}^{N}\widehat{-}x\right) \widehat{\times }%
y \\
&=&\underset{n=0}{\overset{\infty }{\sum }}\left( 1_{\mathbb{A}}^{N}\widehat{%
-}x\right) ^{n}\widehat{-}\left( 1_{\mathbb{A}}^{N}\widehat{-}x\right) 
\widehat{\times }\underset{n=0}{\overset{\infty }{\sum }}\left( 1_{\mathbb{A}%
}^{N}\widehat{-}x\right) ^{n} \\
&=&1_{\mathbb{A}}^{N}\widehat{+}\underset{n=1}{\overset{\infty }{\sum }}%
\left( 1_{\mathbb{A}}^{N}\widehat{-}x\right) ^{n}\widehat{-}\underset{n=1}{%
\overset{\infty }{\sum }}\left( 1_{\mathbb{A}}^{N}\widehat{-}x\right) ^{n} \\
&=&1_{\mathbb{A}}^{N}.
\end{eqnarray*}%
This completes the proof.
\end{proof}

The following example demonstrates the significance of Proposition 1.

\begin{example}
Consider the \textit{Banach }$%
\mathbb{C}
\left( N\right) -$\textit{algebra }$C_{N}\left( \Omega \right) $ given in
Example 2, the $N-$subalgebra $%
\mathbb{C}
\left( N\right) \left[ z\right] $ given in Example 4 and the functions $f,g$
in Example 5. We get%
\begin{eqnarray*}
\overset{..}{\parallel }f-1_{C_{N}\left( \Omega \right) }^{N}\overset{..}{%
\parallel }_{C_{N}\left( \Omega \right) } &=&\underset{z\in \Omega }{\max }%
\overset{..}{\parallel }\left( f-1_{C_{N}\left( \Omega \right) }^{N}\right)
\left( z\right) \overset{..}{\parallel } \\
&=&\underset{z\in \Omega }{\max }\overset{..}{\parallel }z\overset{..}{%
\parallel }=\frac{\overset{..}{1}}{\overset{..}{2}}\beta \overset{..}{<}%
\overset{..}{1}
\end{eqnarray*}%
and this inequality says that the function $f$ is $N-$invertible and $%
\underset{n=0}{\overset{\infty }{\sum }}\left( \ominus z\right)
^{n}=f^{-1_{N}}\left( z\right) $ by Proposition 1. We have shown that $%
f^{-1_{N}}=g$ in Example 5, so $\underset{n=0}{\overset{\infty }{\sum }}%
\left( \ominus z\right) ^{n}=\frac{1_{%
\mathbb{C}
\left( N\right) }}{z\oplus 1_{%
\mathbb{C}
\left( N\right) }}.$
\end{example}

\begin{proposition}
The group $\mathbb{A}^{-1_{N}}$ is an open subset of $\mathbb{A}.$ If $%
\overset{..}{\parallel }x\widehat{-}x_{0}\overset{..}{\parallel }_{\mathbb{A}%
}\overset{..}{<}\frac{\overset{..}{1}}{\overset{..}{\parallel }x_{0}^{-1_{N}}%
\overset{..}{\parallel }_{\mathbb{A}}}\beta $ for an element $x_{0}\in 
\mathbb{A}^{-1_{N}},$ then $x$ is $N-$invertible and $x^{-1_{N}}$ is
represented by%
\begin{equation*}
x^{-1_{N}}=\left( \underset{n=0}{\overset{\infty }{\sum }}\left[
x_{0}^{-1_{N}}\widehat{\times }\left( x_{0}\widehat{-}x\right) \right]
^{n}\right) \widehat{\times }x_{0}^{-1_{N}}.
\end{equation*}
\end{proposition}

\begin{proof}
Let $x_{0}\in \mathbb{A}^{-1_{N}}$ and $\overset{..}{\parallel }x\widehat{-}%
x_{0}\overset{..}{\parallel }_{\mathbb{A}}\overset{..}{<}\frac{\overset{..}{1%
}}{\overset{..}{\parallel }x_{0}^{-1_{N}}\overset{..}{\parallel }_{\mathbb{A}%
}}\beta .$ Then, we have%
\begin{equation*}
\overset{..}{\parallel }I\widehat{-}x_{0}^{-1_{N}}\widehat{\times }x\overset{%
..}{\parallel }_{\mathbb{A}}=\overset{..}{\parallel }x_{0}^{-1_{N}}\widehat{%
\times }\left( x_{0}\widehat{-}x\right) \overset{..}{\parallel }_{\mathbb{A}}%
\overset{..}{\leq }\overset{..}{\parallel }x_{0}^{-1_{N}}\overset{..}{%
\parallel }_{\mathbb{A}}\overset{..}{\times }\overset{..}{\parallel }x_{0}%
\widehat{-}x\overset{..}{\parallel }_{\mathbb{A}}\overset{..}{<}\overset{..}{%
1}.
\end{equation*}%
This implies that%
\begin{equation*}
\left( x_{0}^{-1_{N}}\widehat{\times }x\right) ^{-1_{N}}=\underset{n=0}{%
\overset{\infty }{\sum }}\left[ 1_{\mathbb{A}}^{N}\widehat{-}\left(
x_{0}^{-1_{N}}\widehat{\times }x\right) \right] ^{n},
\end{equation*}%
and so,%
\begin{equation*}
x^{-1_{N}}\widehat{\times }x_{0}=\underset{n=0}{\overset{\infty }{\sum }}%
\left[ x_{0}^{-1_{N}}\widehat{\times }\left( x_{0}\widehat{-}x\right) \right]
^{n}.
\end{equation*}

Hence it follows that $x^{-1_{N}}=\left( \underset{n=0}{\overset{\infty }{%
\sum }}\left[ x_{0}^{-1_{N}}\widehat{\times }\left( x_{0}\widehat{-}x\right) %
\right] ^{n}\right) \widehat{\times }x_{0}^{-1_{N}}.$
\end{proof}

\begin{corollary}
The function $\mathbb{A}^{-1_{N}}\rightarrow \mathbb{A}^{-1_{N}},x%
\rightarrow $ $x^{-1_{N}}$ is continuous in norm $\overset{..}{\parallel }.%
\overset{..}{\parallel }_{\mathbb{A}}$.
\end{corollary}

\begin{proof}
We have by previous proposition and Example 4 in \cite{7} 
\begin{eqnarray*}
\overset{..}{\parallel }x^{-1_{N}}\widehat{-}x_{0}^{-1_{N}}\overset{..}{%
\parallel }_{\mathbb{A}} &=&\overset{..}{\parallel }\left( \underset{n=0}{%
\overset{\infty }{\sum }}\left[ x_{0}^{-1_{N}}\widehat{\times }\left( x_{0}%
\widehat{-}x\right) \right] ^{n}\right) \widehat{\times }x_{0}^{-1_{N}}%
\widehat{-}x_{0}^{-1_{N}}\overset{..}{\parallel }_{\mathbb{A}} \\
&&\overset{..}{\leq }\overset{..}{\parallel }x_{0}^{-1_{N}}\overset{..}{%
\parallel }_{\mathbb{A}}\overset{..}{\times }\underset{n=1}{\overset{\infty }%
{\overset{..}{\sum }}}\left( \overset{..}{\parallel }x_{0}^{-1_{N}}\overset{%
..}{\parallel }_{\mathbb{A}}\overset{..}{\times }\overset{..}{\parallel }%
x_{0}\widehat{-}x\overset{..}{\parallel }_{\mathbb{A}}\right) ^{n} \\
&=&\overset{..}{\parallel }x_{0}^{-1_{N}}\overset{..}{\parallel }_{\mathbb{A}%
}\overset{..}{\times }\left[ \frac{\overset{..}{\parallel }x_{0}^{-1_{N}}%
\overset{..}{\parallel }_{\mathbb{A}}\overset{..}{\times }\overset{..}{%
\parallel }x_{0}\widehat{-}x\overset{..}{\parallel }_{\mathbb{A}}}{\overset{%
..}{1}\overset{..}{-}\overset{..}{\parallel }x_{0}^{-1_{N}}\overset{..}{%
\parallel }_{\mathbb{A}}\overset{..}{\times }\overset{..}{\parallel }x_{0}%
\widehat{-}x\overset{..}{\parallel }_{\mathbb{A}}}\beta \right] \\
&&\overset{..}{\leq }\overset{..}{\parallel }x_{0}^{-1_{N}}\overset{..}{%
\parallel }_{\mathbb{A}}^{\overset{..}{2}}\overset{..}{\times }\overset{..}{%
\parallel }x_{0}\widehat{-}x\overset{..}{\parallel }_{\mathbb{A}}
\end{eqnarray*}%
and the result follows as required.
\end{proof}

Now, we give a new concept as follows:

\begin{definition}
Let $\mathbb{A}$ and $\mathbb{B}$ be $%
\mathbb{C}
\left( N\right) -$\textit{algebras endowed with the algebraic operations
(addition, scalar multiplication, multiplication) }$\widehat{+},\widehat{%
\cdot },\widehat{\times }$ and $\widehat{\widehat{+}},\widehat{\widehat{%
\cdot }},\widehat{\widehat{\times }}$ respectively, \textit{and }$\varphi
_{N}:\mathbb{A}\rightarrow \mathbb{B}$ be a mapping. If $\varphi _{N}\left( x%
\widehat{+}\lambda \widehat{\cdot }y\right) =\varphi _{N}\left( x\right) 
\widehat{\widehat{\times }}\lambda \widehat{\widehat{\cdot }}\varphi
_{N}\left( y\right) $ for all $x,y\in \mathbb{A}$ and for all $\lambda \in 
\mathbb{C}
\left( N\right) $, then the mapping $\varphi _{N}$ is called a $%
\mathbb{C}
\left( N\right) -$\textit{linear map. Also, a }$%
\mathbb{C}
\left( N\right) -$algebra homomorphism $\varphi _{N}:\mathbb{A}\rightarrow 
\mathbb{B}$ is a $%
\mathbb{C}
\left( N\right) -$linear map such that $\varphi _{N}\left( x\widehat{\times }%
y\right) =\varphi _{N}\left( x\right) \widehat{\widehat{\times }}\varphi
_{N}\left( y\right) $ \textit{\ for all }$x,y\in \mathbb{A}$. \textit{If }$%
\mathbb{B}=%
\mathbb{C}
\left( N\right) ,$ then $%
\mathbb{C}
\left( N\right) -$\textit{linear map} $\varphi _{N}$ is called a $%
\mathbb{C}
\left( N\right) -$\textit{linear functional. }
\end{definition}

\begin{definition}
Let $\mathbb{A}$ and $\mathbb{B}$ be $%
\mathbb{C}
\left( N\right) -$\textit{algebras and }$\varphi _{N}:\mathbb{A}\rightarrow 
\mathbb{B}$ be a $%
\mathbb{C}
\left( N\right) -$algebra homomorphism. Then the $N-$kernel of $\varphi _{N}$
is defined by $Ker_{N}\left( \varphi _{N}\right) =\left\{ x\in \mathbb{A}%
:\varphi _{N}\left( x\right) =0_{\mathbb{B}}^{N}\right\} .$
\end{definition}

\begin{remark}
We say $\varphi _{N}$ is $N-$unital if $\mathbb{A}$ and $\mathbb{B}$ are
unital and $\varphi _{N}\left( 1_{\mathbb{A}}^{N}\right) =1_{\mathbb{B}}^{N}$%
.
\end{remark}

\begin{example}
Consider the \textit{Banach }$%
\mathbb{C}
\left( N\right) -$\textit{algebra }$C_{N}\left( \Omega \right) $ given in
Example 2 for $\Omega =\left\{ z\in 
\mathbb{C}
\left( N\right) :\overset{..}{\parallel }z\overset{..}{\parallel }\overset{..%
}{\leq }\frac{\overset{..}{1}}{\overset{..}{2}}\beta \right\} $ and $N-$%
subalgebra $I_{z}$ given in Example 3. Define the mapping%
\begin{equation*}
\varphi _{N}:C_{N}\left( \Omega \right) \rightarrow 
\mathbb{C}
\left( N\right) ,\text{ \ }\varphi _{N}\left( f\right) =f\left( 0_{%
\mathbb{C}
\left( N\right) }\right) .
\end{equation*}%
It is easy to show that%
\begin{eqnarray*}
\varphi _{N}\left( f+\lambda \cdot g\right) &=&\varphi _{N}\left( f\right)
\oplus \lambda \odot \varphi _{N}\left( g\right) , \\
\varphi _{N}\left( f\times g\right) &=&\varphi _{N}\left( f\right) \odot
\varphi _{N}\left( g\right) ,
\end{eqnarray*}%
for all $f,g\in C_{N}\left( \Omega \right) $ and for all $\lambda \in 
\mathbb{C}
\left( N\right) .$ Then, the mapping $\varphi _{N}$ is a $%
\mathbb{C}
\left( N\right) -$\textit{linear functional. }Also, we get%
\begin{eqnarray*}
Ker_{N}\left( \varphi _{N}\right) &=&\left\{ f\in C_{N}\left( \Omega \right)
:\varphi _{N}\left( f\right) =0_{%
\mathbb{C}
\left( N\right) }\right\} \\
&=&\left\{ f\in C_{N}\left( \Omega \right) :f\left( 0_{%
\mathbb{C}
\left( N\right) }\right) =0_{%
\mathbb{C}
\left( N\right) }\right\}
\end{eqnarray*}%
and $\varphi _{N}$ is $N-$unital since $\varphi _{N}\left( 1_{C_{N}\left(
\Omega \right) }^{N}\right) =1_{%
\mathbb{C}
\left( N\right) }.$
\end{example}

\begin{proposition}
Let $\mathbb{A}$ be a $%
\mathbb{C}
\left( N\right) -$\textit{algebra }with a unit $1_{\mathbb{A}}^{N}.$ If $%
\varphi _{N}:\mathbb{A}\rightarrow 
\mathbb{C}
\left( N\right) $ is a $%
\mathbb{C}
\left( N\right) -$algebra homomorphism, then $\varphi _{N}\left( 1_{\mathbb{A%
}}^{N}\right) =1_{%
\mathbb{C}
\left( N\right) },$ and $\varphi _{N}\left( x\right) \neq 0_{%
\mathbb{C}
\left( N\right) }$ for every $N-$invertible $x\in \mathbb{A}$.
\end{proposition}

\begin{proof}
Let $y\in \mathbb{A}$ and $\varphi _{N}\left( y\right) \neq 0_{%
\mathbb{C}
\left( N\right) }.$ Then, since $\varphi _{N}\left( y\right) =\varphi
_{N}\left( y\widehat{\times }1_{\mathbb{A}}^{N}\right) =\varphi _{N}\left(
y\right) \odot \varphi _{N}\left( 1_{\mathbb{A}}^{N}\right) ,$ we have that $%
\varphi _{N}\left( 1_{\mathbb{A}}^{N}\right) =1_{%
\mathbb{C}
\left( N\right) }.$ On the other hand, if $x$ is $N-$invertible, then%
\begin{equation*}
\varphi _{N}\left( x\right) \odot \varphi _{N}\left( x^{-1_{N}}\right)
=\varphi _{N}\left( x\widehat{\times }x^{-1_{N}}\right) =\varphi _{N}\left(
1_{\mathbb{A}}^{N}\right) =1_{%
\mathbb{C}
\left( N\right) }.
\end{equation*}%
Hence it follows that $\varphi _{N}\left( x\right) \neq 0_{%
\mathbb{C}
\left( N\right) }.$ The proof is completed.
\end{proof}

\begin{theorem}
Let $\mathbb{A}$ be a Banach $%
\mathbb{C}
\left( N\right) -$\textit{algebra, }$x\in \mathbb{A}$ and $\overset{..}{%
\parallel }x\overset{..}{\parallel }_{\mathbb{A}}\overset{..}{<}\overset{..}{%
1}.$ Then, $\overset{..}{\parallel }\varphi _{N}\left( x\right) \overset{..}{%
\parallel }\overset{..}{<}\overset{..}{1}$ for every $%
\mathbb{C}
\left( N\right) -$algebra homomorphism $\varphi _{N}:\mathbb{A}\rightarrow 
\mathbb{C}
\left( N\right) $ on $\mathbb{A}$.
\end{theorem}

\begin{proof}
Let $\lambda \in 
\mathbb{C}
\left( N\right) $ and $\overset{..}{\parallel }\lambda \overset{..}{%
\parallel }\overset{..}{\geq }\overset{..}{1}.$ Then, we say that $1_{%
\mathbb{A}}^{N}\widehat{-}\lambda ^{-1}\widehat{\cdot }x$ is $N-$invertible
by Proposition 1. So, we get by Proposition 3%
\begin{equation*}
1_{%
\mathbb{C}
\left( N\right) }\ominus \lambda ^{-1}\odot \varphi \left( x\right) =\varphi
_{N}\left( 1_{\mathbb{A}}^{N}\widehat{-}\lambda ^{-1}\widehat{\cdot }%
x\right) \neq 0_{%
\mathbb{C}
\left( N\right) }.
\end{equation*}%
This implies that $\varphi _{N}\left( x\right) \neq \lambda .$ Therefore, we
obtain that $\overset{..}{\parallel }\varphi _{N}\left( x\right) \overset{..}%
{\parallel }\overset{..}{<}\overset{..}{1}$ for every $%
\mathbb{C}
\left( N\right) -$algebra homomorphism $\varphi _{N}:\mathbb{A}\rightarrow 
\mathbb{C}
\left( N\right) $ on $\mathbb{A}$. This completes the proof.
\end{proof}

\begin{corollary}
All $%
\mathbb{C}
\left( N\right) -$algebra homomorphisms $\varphi _{N}:\mathbb{A}\rightarrow 
\mathbb{C}
\left( N\right) $ of Banach $%
\mathbb{C}
\left( N\right) -$\textit{algebras are continuous.}
\end{corollary}

In the rest of this section, we define the notion of an $N-$ideal and
discuss some relevant findings.

\begin{definition}
Let $\mathbb{A}$ be a $%
\mathbb{C}
\left( N\right) -$\textit{algebra and }$I$ be an $N-$subalgebra of $\mathbb{A%
}$. If $x\widehat{\times }y\in I$ (respectively, $y\widehat{\times }x\in I$)
whenever $x\in \mathbb{A}$ and $y\in I,$ then $I$ is said to be a left
(respectively, right) $N-$ideal of $\mathbb{A}$. An $N-$ideal of $\mathbb{A}$
is an $N-$subalgebra that is simultaneously a left and a right $N-$ideal of $%
\mathbb{A}$. Obviously, $\left\{ 0_{\mathbb{A}}\right\} $ and $\mathbb{A}$
are $N-$ideals in $\mathbb{A}$. An $N-$ideal $I\neq \mathbb{A}$ is called an 
$N-$proper ideal of $\mathbb{A}$. An $N-$maximal ideal in $\mathbb{A}$ is an 
$N-$proper ideal that is not contained in any $N-$ideal except $I$ itself
and the entire $%
\mathbb{C}
\left( N\right) -$\textit{algebra} $\mathbb{A}$. An $N-$ideal $I$ in $%
\mathbb{A}$ is $N-$modular if there exists an element $y\in \mathbb{A}$ such
that the two sets%
\begin{equation*}
\mathbb{A}\widehat{\times }\left( 1_{\mathbb{A}}^{N}\widehat{-}y\right)
:=\left\{ x\widehat{-}x\widehat{\times }y:x\in \mathbb{A}\right\} \text{ and 
}\left( 1_{\mathbb{A}}^{N}\widehat{-}y\right) \widehat{\times }\mathbb{A}%
:=\left\{ x\widehat{-}y\widehat{\times }x:x\in \mathbb{A}\right\}
\end{equation*}%
are both in $I.$
\end{definition}

\begin{example}
Consider the \textit{Banach }$%
\mathbb{C}
\left( N\right) -$\textit{algebra }$C_{N}\left( \Omega \right) $ given in
Example 2 for $\Omega =\left\{ z\in 
\mathbb{C}
\left( N\right) :\overset{..}{\parallel }z\overset{..}{\parallel }\overset{..%
}{\leq }\frac{\overset{..}{1}}{\overset{..}{2}}\beta \right\} $ and $N-$%
subalgebra $I_{z}$ given in Example 3. Since 
\begin{equation*}
\left( f\times g\right) \left( z\right) =f\left( z\right) \odot g\left(
z\right) =f\left( z\right) \odot 0_{%
\mathbb{C}
\left( N\right) }=0_{%
\mathbb{C}
\left( N\right) }
\end{equation*}%
and similarly $\left( g\times f\right) \left( z\right) =0_{%
\mathbb{C}
\left( N\right) }$ for all $f\in C_{N}\left( \Omega \right) $ and for all $%
g\in I_{z},$ we can write $f\widehat{\times }g\in I_{z}$ and $g\widehat{%
\times }f\in I_{z}.$ Therefore, we obtain that $I_{z}$ is $N-$ideal of $%
C_{N}\left( \Omega \right) .$ In addition, since $I_{z}\neq C_{N}\left(
\Omega \right) ,$ $I_{z}$ is $N-$proper.
\end{example}

\begin{example}
Consider the \textit{Banach }$%
\mathbb{C}
\left( N\right) -$\textit{algebra }$C_{N}\left( \Omega \right) $ given in
Example 2 for $\Omega =\left\{ z\in 
\mathbb{C}
\left( N\right) :\overset{..}{\parallel }z\overset{..}{\parallel }\overset{..%
}{\leq }\frac{\overset{..}{1}}{\overset{..}{2}}\beta \right\} $ and $N-$%
ideal $I_{z}$ given in Example 8. We claim that $I_{z}$ is an $N-$maximal
ideal of $C_{N}\left( \Omega \right) .$ Suppose that $I_{z}$ is contained in
some larger $N-$ideals $J$. We will show that $J=C_{N}\left( \Omega \right) $
so then $I_{z}$ is an $N-$maximal ideal. Let $g\in J-I_{z}.$ Then, if we
define the mapping $h:\Omega \rightarrow 
\mathbb{C}
\left( N\right) ,$ $\ h\left( w\right) :=\frac{g\left( w\right) }{g\left(
z\right) },$ we obtain that $h\in J$ and $1_{C_{N}\left( \Omega \right)
}^{N}-h\in I_{z}$. This implies that $h+1_{C_{N}\left( \Omega \right)
}^{N}-h=1_{C_{N}\left( \Omega \right) }^{N}\in J.$ Thus, $J=C_{N}\left(
\Omega \right) $ and so $I_{z}$ is a $N-$maximal ideal.
\end{example}

\begin{theorem}
Let $\mathbb{A}$ and $\mathbb{B}$ be $%
\mathbb{C}
\left( N\right) -$\textit{algebras and }$\varphi _{N}:\mathbb{A}\rightarrow 
\mathbb{B}$ be a $%
\mathbb{C}
\left( N\right) -$algebra homomorphism. Then, $Ker_{N}\left( \varphi
_{N}\right) $ is an $N-$ideal of $\mathbb{A}$.
\end{theorem}

\begin{proof}
Firstly, we show that $Ker_{N}\left( \varphi _{N}\right) $ is an $N-$%
subalgebra of $\mathbb{A}$. Since $\varphi _{N}\left( 0_{\mathbb{A}%
}^{N}\right) =0_{\mathbb{B}}^{N},$ we write $0_{\mathbb{A}}^{N}\in
Ker_{N}\left( \varphi _{N}\right) .$ Since%
\begin{equation*}
\varphi _{N}\left( x\widehat{+}y\right) =\varphi _{N}\left( x\right) 
\widehat{\widehat{+}}\varphi _{N}\left( y\right) =0_{\mathbb{B}}^{N}\widehat{%
\widehat{+}}0_{\mathbb{B}}^{N}=0_{\mathbb{B}}^{N},
\end{equation*}%
\begin{equation*}
\varphi _{N}\left( \lambda \widehat{\cdot }x\right) =\lambda \widehat{%
\widehat{\cdot }}\varphi _{N}\left( x\right) =\lambda \widehat{\widehat{%
\cdot }}0_{\mathbb{B}}^{N}=0_{\mathbb{B}}^{N}
\end{equation*}%
and%
\begin{equation*}
\varphi _{N}\left( x\widehat{\times }y\right) =\varphi _{N}\left( x\right) 
\widehat{\widehat{\times }}\varphi _{N}\left( y\right) =0_{\mathbb{B}}^{N}%
\widehat{\widehat{\times }}0_{\mathbb{B}}^{N}=0_{\mathbb{B}}^{N}
\end{equation*}%
for all $x,y\in Ker_{N}\left( \varphi _{N}\right) $ and for all $\lambda \in 
\mathbb{C}
\left( N\right) ,$ we obtain that $Ker_{N}\left( \varphi _{N}\right) $ is an 
$N-$subalgebra of $\mathbb{A}$.

Now, we prove that $Ker_{N}\left( \varphi _{N}\right) $ is a left and a
right $N-$ideal of $\mathbb{A}$. Since%
\begin{equation*}
\varphi _{N}\left( x\widehat{\times }y\right) =\varphi _{N}\left( x\right) 
\widehat{\widehat{\times }}\varphi _{N}\left( y\right) =\varphi _{N}\left(
x\right) \widehat{\widehat{\times }}0_{\mathbb{B}}^{N}=0_{\mathbb{B}}^{N}
\end{equation*}%
and%
\begin{equation*}
\varphi _{N}\left( y\widehat{\times }x\right) =\varphi _{N}\left( y\right) 
\widehat{\widehat{\times }}\varphi _{N}\left( x\right) =0_{\mathbb{B}}^{N}%
\widehat{\widehat{\times }}\varphi _{N}\left( x\right) =0_{\mathbb{B}}^{N}
\end{equation*}%
for all $x\in \mathbb{A}$ and for all $y\in Ker_{N}\left( \varphi
_{N}\right) ,$ we obtain that $Ker_{N}\left( \varphi _{N}\right) $ is both a
left and a right $N-$ideal of $\mathbb{A}$. This completes the proof.
\end{proof}

\begin{theorem}
Let $\mathbb{A}$ and $\mathbb{B}$ be $%
\mathbb{C}
\left( N\right) -$\textit{algebras and }$\varphi _{N}:\mathbb{A}\rightarrow 
\mathbb{B}$ be a $%
\mathbb{C}
\left( N\right) -$algebra homomorphism. Then, $\varphi _{N}\left( \mathbb{A}%
\right) $ is an $N-$subalgebra of $\mathbb{B}$.
\end{theorem}

\begin{proof}
Let $y_{1},y_{2}$ be arbitrary elements of $\varphi _{N}\left( \mathbb{A}%
\right) .$ Then, there exists $x_{1},x_{2}\in \mathbb{A}$ such that $%
y_{1}=\varphi _{N}\left( x_{1}\right) ,y_{2}=\varphi _{N}\left( x_{2}\right)
.$ So, since%
\begin{equation*}
y_{1}\widehat{\widehat{+}}y_{2}=\varphi _{N}\left( x_{1}\right) \widehat{%
\widehat{+}}\varphi _{N}\left( x_{2}\right) =\varphi _{N}\left( x_{1}%
\widehat{+}x_{2}\right) ,
\end{equation*}%
\begin{equation*}
\lambda \widehat{\widehat{\cdot }}y_{1}=\lambda \widehat{\widehat{\cdot }}%
\varphi _{N}\left( x_{1}\right) =\varphi _{N}\left( \lambda \widehat{\cdot }%
x_{1}\right)
\end{equation*}%
and%
\begin{equation*}
y_{1}\widehat{\widehat{\times }}y_{2}=\varphi _{N}\left( x_{1}\right) 
\widehat{\widehat{\times }}\varphi _{N}\left( x_{2}\right) =\varphi
_{N}\left( x_{1}\widehat{\times }x_{2}\right) ,
\end{equation*}%
then $y_{1}\widehat{\widehat{+}}y_{2},\lambda \widehat{\widehat{\cdot }}%
y_{1},y_{1}\widehat{\widehat{\times }}y_{2}\in $ $\varphi _{N}\left( \mathbb{%
A}\right) .$ This implies that $\varphi _{N}\left( \mathbb{A}\right) $ is an 
$N-$subalgebra of $\mathbb{B}$.
\end{proof}

The following example explains the validity of Theorem 3 and Theorem 4.

\begin{example}
Consider the \textit{Banach }$%
\mathbb{C}
\left( N\right) -$\textit{algebra }$C_{N}\left( \Omega \right) $ given in
Example 2 for $\Omega =\left\{ z\in 
\mathbb{C}
\left( N\right) :\overset{..}{\parallel }z\overset{..}{\parallel }\overset{..%
}{\leq }\frac{\overset{..}{1}}{\overset{..}{2}}\beta \right\} $ and the $%
\mathbb{C}
\left( N\right) -$\textit{linear functional} $\varphi _{N}$ given in Example
7. Then, $Ker_{N}\left( \varphi _{N}\right) =\left\{ f\in C_{N}\left( \Omega
\right) :f\left( 0_{%
\mathbb{C}
\left( N\right) }\right) =0_{%
\mathbb{C}
\left( N\right) }\right\} $ is an $N-$ideal of $C_{N}\left( \Omega \right) $
and $\varphi _{N}\left( C_{N}\left( \Omega \right) \right) $ is an $N-$%
subalgebra of $%
\mathbb{C}
\left( N\right) .$
\end{example}

\begin{remark}
Let $\mathbb{A}$ be a Banach $%
\mathbb{C}
\left( N\right) -$\textit{algebra and }$I$ be an $N-$ideal of $\mathbb{A}$.
Then, $\mathbb{A}/I$ is an $%
\mathbb{C}
\left( N\right) -$\textit{algebra with the }$N-$multiplication defined by%
\begin{equation*}
\left( x\widehat{+}I\right) \widehat{\times }\left( y\widehat{+}I\right) =x%
\widehat{\times }y\widehat{+}I
\end{equation*}%
for all $x\widehat{+}I,y\widehat{+}I\in \mathbb{A}/I.$ Also, $I$ is $N-$%
modular if and only if $\mathbb{A}/I$ is unital.
\end{remark}

\begin{theorem}
If $I$ is a closed $N-$ideal in a normed $%
\mathbb{C}
\left( N\right) -$\textit{algebra }$\mathbb{A}$, then $\mathbb{A}/I$ is a
normed $%
\mathbb{C}
\left( N\right) -$\textit{algebra with the }$N-$quotient norm%
\begin{equation*}
\overset{..}{\parallel }x\widehat{+}I\overset{..}{\parallel }_{\mathbb{A}/I}=%
\overset{..}{\inf }\left\{ \overset{..}{\parallel }x\widehat{+}y\overset{..}{%
\parallel }_{\mathbb{A}}:y\in I\right\} .
\end{equation*}
\end{theorem}

\begin{proof}
Let $\varepsilon \overset{..}{>}\overset{..}{0}$ and $x,y\in \mathbb{A}$.
Then, there exist $x_{0},y_{0}\in I$ such that $\varepsilon \overset{..}{+}%
\overset{..}{\parallel }x\widehat{+}I\overset{..}{\parallel }_{\mathbb{A}/I}%
\overset{..}{>}\overset{..}{\parallel }x\widehat{+}x_{0}\overset{..}{%
\parallel }_{\mathbb{A}}$ and $\varepsilon \overset{..}{+}\overset{..}{%
\parallel }y\widehat{+}I\overset{..}{\parallel }_{\mathbb{A}/I}\overset{..}{>%
}\overset{..}{\parallel }y\widehat{+}y_{0}\overset{..}{\parallel }_{\mathbb{A%
}}$ by Theorem 13 (ii) in \cite{8}. Put $z=x_{0}\widehat{\times }y\widehat{+}%
x\widehat{\times }y_{0}\widehat{+}x_{0}\widehat{\times }y_{0}.$ It is clear
that $z\in I.$ Therefore, we can write%
\begin{equation*}
\left( \varepsilon \overset{..}{+}\overset{..}{\parallel }x\widehat{+}I%
\overset{..}{\parallel }_{\mathbb{A}/I}\right) \overset{..}{\times }\left(
\varepsilon \overset{..}{+}\overset{..}{\parallel }y\widehat{+}I\overset{..}{%
\parallel }_{\mathbb{A}/I}\right) \overset{..}{>}\overset{..}{\parallel }x%
\widehat{+}x_{0}\overset{..}{\parallel }_{\mathbb{A}}\overset{..}{\times }%
\overset{..}{\parallel }y\widehat{+}y_{0}\overset{..}{\parallel }_{\mathbb{A}%
}\overset{..}{\geq }\overset{..}{\parallel }x\widehat{\times }y\widehat{+}z%
\overset{..}{\parallel }_{\mathbb{A}}\overset{..}{\geq }\overset{..}{%
\parallel }x\widehat{\times }y\widehat{+}I\overset{..}{\parallel }_{\mathbb{A%
}/I}.
\end{equation*}%
Taking $\varepsilon \rightarrow \overset{..}{0},$ it follows that $\overset{%
..}{\parallel }x\widehat{+}I\overset{..}{\parallel }_{\mathbb{A}/I}\overset{%
..}{\times }\overset{..}{\parallel }y\widehat{+}I\overset{..}{\parallel }_{%
\mathbb{A}/I}\overset{..}{\geq }\overset{..}{\parallel }x\widehat{\times }y%
\widehat{+}I\overset{..}{\parallel }_{\mathbb{A}/I}.$ This implies that $%
\mathbb{A}/I$ is a normed $%
\mathbb{C}
\left( N\right) -$algebra. The proof is completed.
\end{proof}

\begin{theorem}
If $I$ is an $N-$ideal in a normed $%
\mathbb{C}
\left( N\right) -$\textit{algebra }$\mathbb{A}$, then the $N-$quotient map $%
\pi :\mathbb{A}\rightarrow \mathbb{A}/I$ is a $%
\mathbb{C}
\left( N\right) -$linear map.
\end{theorem}

\begin{proof}
We have for all $x,y\in \mathbb{A}$%
\begin{equation*}
\pi \left( x\widehat{+}\lambda \widehat{\cdot }y\right) =\left( x\widehat{+}%
\lambda \widehat{\cdot }y\right) \widehat{+}I=\left( x\widehat{+}I\right) 
\widehat{\times }\left( \lambda \widehat{\cdot }y\widehat{+}I\right) =\left(
x\widehat{+}I\right) \widehat{\times }\lambda \widehat{\cdot }\left( y%
\widehat{+}I\right) =\pi \left( x\right) \widehat{+}\lambda \widehat{\cdot }%
\pi \left( y\right) .
\end{equation*}%
This shows that $\pi :\mathbb{A}\rightarrow \mathbb{A}/I$ is a $%
\mathbb{C}
\left( N\right) -$linear map.
\end{proof}

\begin{theorem}
Let $I$ be an $N-$modular ideal of a Banach $%
\mathbb{C}
\left( N\right) -$\textit{algebra }$\mathbb{A}$ with a unit $1_{\mathbb{A}%
}^{N}$. If $I$ is $N-$proper, so is its non-Newtonian closure $\overline{I}%
^{N}.$ If $I$ is $N-$maximal, it is non-Newtonian closed.
\end{theorem}

\begin{proof}
Let $a$ be an element of $\mathbb{A}$ such that $x\widehat{-}x\widehat{%
\times }a,$ $x\widehat{-}a\widehat{\times }x\in I$ for all $x\in \mathbb{A}$%
. If $b\in I$ and $\overset{..}{\parallel }a\widehat{-}b\overset{..}{%
\parallel }_{\mathbb{A}}\overset{..}{<}\overset{..}{1},$ then $y=1_{\mathbb{A%
}}^{N}\widehat{-}\left( a\widehat{+}b\right) $ is $N-$invertible in $\mathbb{%
A}$ by Proposition 1. If $x\in \mathbb{A}$, then $x\widehat{\times }y=x%
\widehat{-}x\widehat{\times }a\widehat{+}x\widehat{\times }b\in I$ and $y\in 
\mathbb{A}^{-1_{N}},$ and hence $\mathbb{A}=\mathbb{A}\widehat{\times }%
y\subset I$ which contradicts the assumption that $I$ is $N-$proper. Then,
this shows that $\overset{..}{\parallel }x\widehat{-}b\overset{..}{\parallel 
}_{\mathbb{A}}\overset{..}{\geq }\overset{..}{1}$ for all $b\in I.$
Therefore, we deduce that $x\notin \overline{I}^{N}$ and so $\overline{I}%
^{N} $ is $N-$proper.

If $I$ is $N-$maximal, then $\overline{I}^{N}$ is a $N-$proper ideal
containing $I$ and this implies that $\overline{I}^{N}=I.$
\end{proof}

\subsection{Non-Newtonian $C^{\ast }-$Algebras}

This section is fundamentals of Non-Newtonian $C^{\ast }-$Algebras.

At first, we present the definition of a non-Newtonian $C^{\ast }-$algebra.

\begin{definition}
Let $\mathbb{A}$ be a $%
\mathbb{C}
\left( N\right) -$\textit{algebra. An }$N-$involution on $\mathbb{A}$ is a
mapping $^{\ast _{N}}:\mathbb{A}\rightarrow \mathbb{A},$ $x\rightarrow
x^{\ast _{N}}$ satisfying the following conditions:
\end{definition}

\textit{(i) }$\left( x\widehat{+}y\right) ^{\ast _{N}}=x^{\ast _{N}}\widehat{%
+}y^{\ast _{N}},$

\textit{(ii) }$\left( \lambda \widehat{\cdot }x\right) ^{\ast _{N}}=%
\overline{\lambda }^{N}\widehat{\cdot }x^{\ast _{N}},$

\textit{(iii) }$\left( x\widehat{\times }y\right) ^{\ast _{N}}=y^{\ast _{N}}%
\widehat{\times }x^{\ast _{N}},$

\textit{(iv) }$\left( x^{\ast _{N}}\right) ^{\ast _{N}}=x$

\textit{for all }$x,y\in \mathbb{A}$\textit{\ and for all }$\lambda \in 
\mathbb{C}
\left( N\right) .$\textit{\ Then, }$\mathbb{A}$\textit{\ is called a }$\ast
_{N}-$\textit{algebra (non-Newtonian }$\ast -$\textit{algebra).}

\textit{If }$\ast _{N}-$\textit{algebra }$\mathbb{A}$\textit{\ is a normed }$%
\mathbb{C}
\left( N\right) -$\textit{algebra and the }$N-$\textit{involution is
isometric, that is, }$\overset{..}{\parallel }x^{\ast _{N}}\overset{..}{%
\parallel }_{\mathbb{A}}=\overset{..}{\parallel }x\overset{..}{\parallel }_{%
\mathbb{A}}$\textit{\ for all }$x\in \mathbb{A}$\textit{, then }$\mathbb{A}$%
\textit{\ is called a normed }$\ast _{N}-$\textit{algebra (non-Newtonian
normed }$\ast -$\textit{algebra).}

\textit{If }$\ast _{N}-$\textit{algebra }$\mathbb{A}$\textit{\ is a Banach }$%
\mathbb{C}
\left( N\right) -$\textit{algebra and the }$N-$\textit{involution is
isometric, then }$\mathbb{A}$\textit{\ is called a Banach }$\ast _{N}-$%
\textit{algebra (non-Newtonian Banach }$\ast -$\textit{algebra).}

\textit{If }$\ast _{N}-$\textit{algebra }$\mathbb{A}$\textit{\ is a Banach }$%
\mathbb{C}
\left( N\right) -$\textit{algebra and its norm satisfies the equality }$%
\overset{..}{\parallel }x^{\ast _{N}}\widehat{\times }x\overset{..}{%
\parallel }_{\mathbb{A}}=\overset{..}{\parallel }x\overset{..}{\parallel }_{%
\mathbb{A}}^{\overset{..}{2}}$\textit{\ for all }$x\in \mathbb{A}$\textit{,
then }$\mathbb{A}$\textit{\ is called a }$C^{\ast _{N}}-$\textit{algebra
(non-Newtonian }$C^{\ast }-$\textit{algebra).}

\textit{Note that a }$C^{\ast _{N}}-$\textit{algebra is a Banach }$\ast
_{N}- $\textit{algebra because the equaality }$\overset{..}{\parallel }%
x^{\ast _{N}}\widehat{\times }x\overset{..}{\parallel }_{\mathbb{A}}=\overset%
{..}{\parallel }x\overset{..}{\parallel }_{\mathbb{A}}^{\overset{..}{2}}$%
\textit{\ implies }$\overset{..}{\parallel }x\overset{..}{\parallel }_{%
\mathbb{A}}\overset{..}{\leq }\overset{..}{\parallel }x^{\ast _{N}}\overset{%
..}{\parallel }_{\mathbb{A}}$\textit{\ for all }$x\in \mathbb{A}$\textit{\
and substituting }$x^{\ast _{N}}$\textit{\ for }$x$\textit{\ in this
inequality gives that }$\overset{..}{\parallel }x^{\ast _{N}}\overset{..}{%
\parallel }_{\mathbb{A}}\overset{..}{\leq }\overset{..}{\parallel }\left(
x^{\ast _{N}}\right) ^{\ast _{N}}\overset{..}{\parallel }_{\mathbb{A}}=%
\overset{..}{\parallel }x\overset{..}{\parallel }_{\mathbb{A}},$\textit{\
hence }$\overset{..}{\parallel }x^{\ast _{N}}\overset{..}{\parallel }_{%
\mathbb{A}}=\overset{..}{\parallel }x\overset{..}{\parallel }_{\mathbb{A}}$%
\textit{\ for all }$x\in \mathbb{A}$\textit{. Thus, also, the }$N-$\textit{%
involution of }$\mathbb{A}$\textit{\ is continuous.}

\textit{If a }$C^{\ast _{N}}-$\textit{algebra has a unit }$1_{\mathbb{A}%
}^{N},$\textit{\ then automatically }$\overset{..}{\parallel }1_{\mathbb{A}%
}^{N}\overset{..}{\parallel }_{\mathbb{A}}=\overset{..}{1},$\textit{\
because }$\overset{..}{\parallel }1_{\mathbb{A}}^{N}\overset{..}{\parallel }%
_{\mathbb{A}}=\overset{..}{\parallel }1_{\mathbb{A}}^{N\ast _{N}}\widehat{%
\times }1_{\mathbb{A}}^{N}\overset{..}{\parallel }_{\mathbb{A}}=\overset{..}{%
\parallel }1_{\mathbb{A}}^{N}\overset{..}{\parallel }_{\mathbb{A}}^{\overset{%
..}{2}}.$

The followings are examples of $C^{\ast _{N}}-$algebras.

\begin{example}
The Banach $%
\mathbb{C}
\left( N\right) -$algebra $%
\mathbb{C}
\left( N\right) $ forms a $C^{\ast _{N}}-$algebra with the algebraic
operations $\oplus $ and $\odot $, the norm $\overset{..}{\parallel }.%
\overset{..}{\parallel }$ defined on it and the $N-$involution $^{\ast _{N}}:%
\mathbb{C}
\left( N\right) \rightarrow 
\mathbb{C}
\left( N\right) ,$ $z\rightarrow \overline{z}$.
\end{example}

\begin{example}
Consider the \textit{Banach }$%
\mathbb{C}
\left( N\right) -$\textit{algebra }$C_{N}\left( \Omega \right) $ given in
Example 2 for $\Omega =\left\{ z\in 
\mathbb{C}
\left( N\right) :\overset{..}{\parallel }z\overset{..}{\parallel }\overset{..%
}{\leq }\frac{\overset{..}{1}}{\overset{..}{2}}\beta \right\} .$ Define the $%
N-$involution%
\begin{eqnarray*}
^{\ast _{N}} &:&C_{N}\left( \Omega \right) \rightarrow C_{N}\left( \Omega
\right) ,f\rightarrow f^{\ast _{N}} \\
f^{\ast _{N}} &:&\Omega \rightarrow 
\mathbb{C}
\left( N\right) ,\text{ \ }f^{\ast _{N}}\left( z\right) =\overline{f\left(
z\right) }.
\end{eqnarray*}%
With a direct computation it can be seen that%
\begin{eqnarray*}
\left( f+g\right) ^{\ast _{N}} &=&f^{\ast _{N}}+g^{\ast _{N}}, \\
\left( \lambda \cdot f\right) ^{\ast _{N}} &=&\overline{\lambda }^{N}\cdot
f^{\ast _{N}}, \\
\left( f\times g\right) ^{\ast _{N}} &=&g^{\ast _{N}}\times f^{\ast _{N}}, \\
\left( f^{\ast _{N}}\right) ^{\ast _{N}} &=&f
\end{eqnarray*}%
for all $f,g\in C_{N}\left( \Omega \right) $ and for all $\lambda \in 
\mathbb{C}
\left( N\right) .$ Also, we get%
\begin{equation*}
\overset{..}{\parallel }f^{\ast _{N}}\overset{..}{\parallel }_{C_{N}\left(
\Omega \right) }=\underset{z\in \Omega }{\max }\overset{..}{\parallel }%
f^{\ast _{N}}\left( z\right) \overset{..}{\parallel }=\underset{z\in \Omega }%
{\max }\overset{..}{\parallel }\overline{f\left( z\right) }\overset{..}{%
\parallel }=\underset{z\in \Omega }{\max }\overset{..}{\parallel }f\left(
z\right) \overset{..}{\parallel }=\overset{..}{\parallel }f\overset{..}{%
\parallel }_{C_{N}\left( \Omega \right) }
\end{equation*}%
and%
\begin{equation*}
\overset{..}{\parallel }f^{\ast _{N}}\times f\overset{..}{\parallel }%
_{C_{N}\left( \Omega \right) }=\underset{z\in \Omega }{\max }\overset{..}{%
\parallel }\left( f^{\ast _{N}}\times f\right) \left( z\right) \overset{..}{%
\parallel }=\underset{z\in \Omega }{\max }\overset{..}{\parallel }f\left(
z\right) \overset{..}{\parallel }^{\overset{..}{2}}=\overset{..}{\parallel }f%
\overset{..}{\parallel }_{C_{N}\left( \Omega \right) }^{\overset{..}{2}}.
\end{equation*}%
Thus, $C_{N}\left( \Omega \right) $ is a $C^{\ast _{N}}-$\textit{algebra.}
\end{example}

Let us define some special types of elements in a $C^{\ast _{N}}-$algebra.

\begin{definition}
Let $\mathbb{A}$ be a $C^{\ast _{N}}-$\textit{algebra and }$x\in \mathbb{A}$%
. Then,
\end{definition}

\qquad \textit{(i) }$x$\textit{\ is called an }$N-$\textit{self adjoint (}$%
N- $\textit{hermitian) element if }$x=x^{\ast _{N}}.$

\textit{\qquad (ii) }$x$\textit{\ is called an }$N-$\textit{normal element
if }$x\widehat{\times }x^{\ast _{N}}=x^{\ast _{N}}\widehat{\times }x.$

\textit{\qquad (iii) }$x$\textit{\ is called an }$N-$\textit{unitary element
if }$x\widehat{\times }x^{\ast _{N}}=x^{\ast _{N}}\widehat{\times }x=1_{%
\mathbb{A}}^{N}.$

\begin{example}
The set of $N-$self adjoint elements and unitary elements of $C_{N}\left(
\Omega \right) $ is 
\begin{equation*}
\left\{ f:\Omega \rightarrow 
\mathbb{C}
\left( N\right) :f\left( z\right) =\left( \overset{.}{a},\overset{..}{0}%
\right) \text{ for all }z\in \Omega ,\text{ where }a\in 
\mathbb{R}
\right\}
\end{equation*}%
and%
\begin{equation*}
\left\{ f:\Omega \rightarrow 
\mathbb{C}
\left( N\right) :\overset{..}{\parallel }f\left( z\right) \overset{..}{%
\parallel }=1_{%
\mathbb{C}
\left( N\right) }\text{ for all }z\in \Omega \right\} ,
\end{equation*}%
respectively. Since $C_{N}\left( \Omega \right) $ is commutative with
respect to multiplication $\left( \times \right) ,$ the set of $N-$normal
elements of $C_{N}\left( \Omega \right) $ is itself.
\end{example}

\begin{proposition}
\textit{Let }$A$\textit{\ be a }$C^{\ast _{N}}-$\textit{algebra and }$x\in A$%
\textit{. Then,}
\end{proposition}

\textit{\qquad (i) If }$x$\textit{\ is }$N-$\textit{\ invertible, then }$%
x^{\ast _{N}}$\textit{\ is }$N-$\textit{invertible and }$\left( x^{\ast
_{N}}\right) ^{-1_{N}}=\left( x^{-1_{N}}\right) ^{\ast _{N}}.$

\textit{\qquad (ii) }$x=u\widehat{+}i_{%
\mathbb{C}
\left( N\right) }\widehat{\cdot }v,$\textit{\ where }$u$\textit{\ and }$v$%
\textit{\ are }$N-$\textit{hermitian elements of }$A$\textit{.}

\textit{\qquad (iii) If }$x$\textit{\ is }$N-$\textit{unitary element of }$A$%
\textit{, then }$\overset{..}{\parallel }x\overset{..}{\parallel }_{\mathbb{A%
}}=\overset{..}{1}.$

\begin{proof}
(i) If $x$\ is $N-$\ invertible, then there is an element $y$ in $\mathbb{A}$
such that\textit{\ }$x\widehat{\times }y=y\widehat{\times }x=1_{\mathbb{A}%
}^{N}.$ This implies that $\left( x\widehat{\times }y\right) ^{\ast
_{N}}=\left( y\widehat{\times }x\right) ^{\ast _{N}}=\left( 1_{\mathbb{A}%
}^{N}\right) ^{\ast _{N}}$ and hence, $y^{\ast _{N}}\widehat{\times }x^{\ast
_{N}}=x^{\ast _{N}}\widehat{\times }y^{\ast _{N}}=1_{\mathbb{A}}^{N}.$ So,
it follows that $x^{\ast _{N}}$ is $N-$invertible and $\left( x^{\ast
_{N}}\right) ^{-1_{N}}=y^{\ast _{N}}=\left( x^{-1_{N}}\right) ^{\ast _{N}},$
as required.

(ii) Let $u=\frac{1_{%
\mathbb{C}
\left( N\right) }}{\left( \overset{.}{2},\overset{..}{0}\right) \odot 1_{%
\mathbb{C}
\left( N\right) }}\widehat{\cdot }\left( x\widehat{+}x^{\ast _{N}}\right) $
and $v=\frac{1_{%
\mathbb{C}
\left( N\right) }}{\left( \overset{.}{2},\overset{..}{0}\right) \odot i_{%
\mathbb{C}
\left( N\right) }}\widehat{\cdot }\left( x\widehat{-}x^{\ast _{N}}\right) $.
Then, we have%
\begin{eqnarray*}
u^{\ast _{N}} &=&\left( \frac{1_{%
\mathbb{C}
\left( N\right) }}{\left( \overset{.}{2},\overset{..}{0}\right) \odot 1_{%
\mathbb{C}
\left( N\right) }}\widehat{\cdot }\left( x\widehat{+}x^{\ast _{N}}\right)
\right) ^{\ast _{N}}=\overline{\left( \frac{1_{%
\mathbb{C}
\left( N\right) }}{\left( \overset{.}{2},\overset{..}{0}\right) \odot 1_{%
\mathbb{C}
\left( N\right) }}\right) }\widehat{\cdot }\left( x\widehat{+}x^{\ast
_{N}}\right) ^{\ast _{N}} \\
&=&\overline{\left( \frac{\left( \overset{.}{1},\overset{..}{0}\right) }{%
\left( \overset{.}{2},\overset{..}{0}\right) }\right) }\widehat{\cdot }%
\left( x\widehat{+}x^{\ast _{N}}\right) ^{\ast _{N}}=\frac{\left( \overset{.}%
{1},\overset{..}{0}\right) }{\left( \overset{.}{2},\overset{..}{0}\right) }%
\widehat{\cdot }\left( x^{\ast _{N}}\widehat{+}x\right) \\
&=&\frac{1_{%
\mathbb{C}
\left( N\right) }}{\left( \overset{.}{2},\overset{..}{0}\right) \odot 1_{%
\mathbb{C}
\left( N\right) }}\widehat{\cdot }\left( x\widehat{+}x^{\ast _{N}}\right) =u
\end{eqnarray*}%
and%
\begin{eqnarray*}
v^{\ast _{N}} &=&\left( \frac{1_{%
\mathbb{C}
\left( N\right) }}{\left( \overset{.}{2},\overset{..}{0}\right) \odot i_{%
\mathbb{C}
\left( N\right) }}\widehat{\cdot }\left( x\widehat{-}x^{\ast _{N}}\right)
\right) ^{\ast _{N}}=\overline{\left( \frac{1_{%
\mathbb{C}
\left( N\right) }}{\left( \overset{.}{2},\overset{..}{0}\right) \odot i_{%
\mathbb{C}
\left( N\right) }}\right) }\widehat{\cdot }\left( x\widehat{-}x^{\ast
_{N}}\right) ^{\ast _{N}} \\
&=&\overline{\left( \frac{\left( \overset{.}{1},\overset{..}{0}\right) }{%
\left( \overset{.}{0},\overset{..}{2}\right) }\right) }\widehat{\cdot }%
\left( x\widehat{-}x^{\ast _{N}}\right) ^{\ast _{N}}=\frac{\left( \overset{.}%
{1},\overset{..}{0}\right) }{\left( \overset{.}{0},\overset{..}{-}\overset{..%
}{2}\right) }\widehat{\cdot }\left( x^{\ast _{N}}\widehat{-}x\right) \\
&=&\left[ \frac{\left( \overset{.}{1},\overset{..}{0}\right) }{\ominus
_{1}\left( \overset{.}{0},\overset{..}{2}\right) }\otimes _{1}\left( \ominus
_{1}\left( \overset{.}{1},\overset{..}{0}\right) \right) \right] \widehat{%
\cdot }\left( x\widehat{-}x^{\ast _{N}}\right) \\
&=&\frac{1_{%
\mathbb{C}
\left( N\right) }}{\left( \overset{.}{2},\overset{..}{0}\right) \odot i_{%
\mathbb{C}
\left( N\right) }}\widehat{\cdot }\left( x\widehat{-}x^{\ast _{N}}\right) =v.
\end{eqnarray*}%
Also, it is clear that%
\begin{equation*}
u\widehat{+}i_{%
\mathbb{C}
\left( N\right) }\widehat{\cdot }v=\left[ \frac{1_{%
\mathbb{C}
\left( N\right) }}{\left( \overset{.}{2},\overset{..}{0}\right) \odot 1_{%
\mathbb{C}
\left( N\right) }}\widehat{\cdot }\left( x\widehat{+}x^{\ast _{N}}\right) %
\right] \widehat{+}i_{%
\mathbb{C}
\left( N\right) }\widehat{\cdot }\left[ \frac{1_{%
\mathbb{C}
\left( N\right) }}{\left( \overset{.}{2},\overset{..}{0}\right) \odot i_{%
\mathbb{C}
\left( N\right) }}\widehat{\cdot }\left( x\widehat{-}x^{\ast _{N}}\right) %
\right] =x.
\end{equation*}%
The proof is completed.

(iii) If $x$\ is $N-$unitary element of $\mathbb{A}$, then $x\widehat{\times 
}x^{\ast _{N}}=x^{\ast _{N}}\widehat{\times }x=1_{\mathbb{A}}^{N}.$ So, we
can write $\overset{..}{\parallel }x\widehat{\times }x^{\ast _{N}}\overset{..%
}{\parallel }_{\mathbb{A}}=\overset{..}{\parallel }x^{\ast _{N}}\widehat{%
\times }x\overset{..}{\parallel }_{\mathbb{A}}=\overset{..}{\parallel }1_{%
\mathbb{A}}^{N}\overset{..}{\parallel }_{\mathbb{A}}.$ This means that $%
\overset{..}{\parallel }x\overset{..}{\parallel }_{\mathbb{A}}^{\overset{..}{%
2}}=\overset{..}{1}$ and so, $\overset{..}{\parallel }x\overset{..}{%
\parallel }_{\mathbb{A}}=\overset{..}{1}.$
\end{proof}

The following example illustrates Proposition 4 i).

\begin{example}
Consider the $C^{\ast _{N}}-$\textit{algebra }$C_{N}\left( \Omega \right) $
given in Example 12, the $N-$subalgebra $%
\mathbb{C}
\left( N\right) \left[ z\right] $ given in Example 4 and the $N-$invertible
functions $f,g$ in Example 5. Then, we have%
\begin{equation*}
f^{\ast _{N}}\left( z\right) =\overline{f\left( z\right) }=\overline{z+1}=%
\overline{z}+1_{%
\mathbb{C}
\left( N\right) }
\end{equation*}%
and%
\begin{equation*}
\left( f^{\ast _{N}}\right) ^{-1_{N}}\left( z\right) =\frac{1_{%
\mathbb{C}
\left( N\right) }}{\overline{z}+1_{%
\mathbb{C}
\left( N\right) }}=\overline{\left( \frac{1_{%
\mathbb{C}
\left( N\right) }}{z+1_{%
\mathbb{C}
\left( N\right) }}\right) }=\overline{g\left( z\right) }=g^{\ast _{N}}\left(
z\right) =\left( f^{-1}\right) ^{\ast _{N}}\left( z\right) ,
\end{equation*}%
that is, $f^{\ast _{N}}$ is $N-$invertible and $\left( f^{\ast _{N}}\right)
^{-1_{N}}=\left( f^{-1}\right) ^{\ast _{N}}.$
\end{example}

Now we present non-Newtonian version of self adjoint ideals using the
concept of an $N-$involution.

\begin{definition}
Let $\mathbb{A}$ be a $\ast _{N}-$\textit{algebra and }$\mathbb{B}$ be an $%
N- $subalgebra of $\mathbb{A}$. If $\mathbb{B}^{\ast _{N}}=\left\{ x^{\ast
_{N}}:x\in \mathbb{B}\right\} =\mathbb{B},$ then we say $\mathbb{B}$ is $N-$%
self adjoint.
\end{definition}

\begin{definition}
If $\mathbb{A}$ is a $\ast _{N}-$\textit{algebra, }$\mathbb{B}$ is an $N-$%
subalgebra of $\mathbb{A}$ and $\mathbb{B}$ is $N-$self adjoint, then $%
\mathbb{B}$ is called a $\ast -$subalgebra of $\mathbb{A}$.
\end{definition}

\begin{definition}
If $\mathbb{A}$ is a $\ast _{N}-$\textit{algebra, }$I$ is an $N-$ideal of $%
\mathbb{A}$ and $I$ is $N-$self adjoint, then $I$ is called an $N-$self
adjoint ideal of $\mathbb{A}$.
\end{definition}

\begin{example}
Consider the $C^{\ast _{N}}-$\textit{algebra }$C_{N}\left( \Omega \right) $
given in Example 12 for $\Omega =\left\{ z\in 
\mathbb{C}
\left( N\right) :\overset{..}{\parallel }z\overset{..}{\parallel }\overset{..%
}{\leq }\frac{\overset{..}{1}}{\overset{..}{2}}\beta \right\} $ and the $N-$%
ideal $I_{z}$ given in Example 8. Since%
\begin{equation*}
I_{z}^{\ast _{N}}=\left\{ f^{\ast _{N}}:f\in I_{z}\right\}
\end{equation*}%
and $f^{\ast _{N}}\left( z\right) =\overline{f\left( z\right) }=0_{%
\mathbb{C}
\left( N\right) }$, we have $I_{z}^{\ast _{N}}=I_{z}.$ Thus, $I_{z}$ is an $%
N-$self adjoint ideal of $C_{N}\left( \Omega \right) $.
\end{example}

The next definition is non-Newtonian version of a $\ast -$homomorphism.

\begin{definition}
Let $\mathbb{A}$ and $\mathbb{B}$ be $C^{\ast _{N}}-$\textit{algebras and }$%
\varphi _{N}:\mathbb{A}\rightarrow \mathbb{B}$ be a $%
\mathbb{C}
\left( N\right) -$algebra homomorphism. If $\varphi _{N}\left( x^{\ast
_{N}}\right) =\varphi _{N}\left( x\right) ^{\ast _{N}}$ for all $x\in 
\mathbb{A}$, then the homomorphism $\varphi _{N}$ is called a $\ast _{N}-$%
\textit{homomorphism. If \ in addition }$\varphi _{N}$ is a bijection, then
it is a $\ast _{N}-$isomorphism.
\end{definition}

\begin{example}
Consider the $C^{\ast _{N}}-$\textit{algebra }$C_{N}\left( \Omega \right) $
given in Example 12 for $\Omega =\left\{ z\in 
\mathbb{C}
\left( N\right) :\overset{..}{\parallel }z\overset{..}{\parallel }\overset{..%
}{\leq }\frac{\overset{..}{1}}{\overset{..}{2}}\beta \right\} $ and the $%
\mathbb{C}
\left( N\right) -$\textit{linear functional} $\varphi _{N}$ given in Example
7. Then, we get%
\begin{equation*}
\varphi _{N}\left( f^{\ast _{N}}\right) =f^{\ast _{N}}\left( 0_{%
\mathbb{C}
\left( N\right) }\right) =\overline{f\left( 0_{%
\mathbb{C}
\left( N\right) }\right) }=\overline{\varphi _{N}\left( f^{\ast }\right) }%
=\varphi _{N}\left( f\right) ^{\ast _{N}}.
\end{equation*}%
So, $\varphi _{N}$ is a $\ast _{N}-$\textit{homomorphism.}
\end{example}

In the closing of this section, we introduce two theorems about $\ast _{N}-$%
homomorphisms, and also, we give an example which satisfy the requirements
of these theorems.

\begin{theorem}
Let $\mathbb{A}$ and $\mathbb{B}$ be $\ast _{N}-$algebras and $\varphi _{N}$
be a $\ast _{N}-$homomorphism from $\mathbb{A}$ to $\mathbb{B}$. Then, $%
Ker_{N}\left( \varphi _{N}\right) $ is an $N-$self adjoint ideal of $\mathbb{%
A}$.
\end{theorem}

\begin{proof}
We know that $Ker_{N}\left( \varphi _{N}\right) $ is an $N-$ideal of $%
\mathbb{A}$ by Theorem 3. Then, it is enough to show that $Ker_{N}\left(
\varphi _{N}\right) $ is $N-$self adjoint, that is, $\left( Ker_{N}\left(
\varphi _{N}\right) \right) ^{\ast _{N}}=Ker_{N}\left( \varphi _{N}\right) $
holds.

Let $x\in \left( Ker_{N}\left( \varphi _{N}\right) \right) ^{\ast _{N}}.$
Then, $x^{\ast _{N}}\in Ker_{N}\left( \varphi _{N}\right) $ by definition of 
$N-$involution. Hence, $\varphi _{N}\left( x^{\ast _{N}}\right) =\varphi
_{N}\left( x\right) ^{\ast _{N}}=0_{\mathbb{B}}^{N}.$ This implies that $%
\varphi _{N}\left( x\right) =\left( 0_{\mathbb{B}}^{N}\right) ^{\ast
_{N}}=0_{\mathbb{B}}^{N}$ and so $x\in Ker_{N}\left( \varphi _{N}\right) .$
Then, we have $\left( Ker_{N}\left( \varphi _{N}\right) \right) ^{\ast
_{N}}\subset Ker_{N}\left( \varphi _{N}\right) .$

Conversely, let $x\in Ker_{N}\left( \varphi _{N}\right) .$ Then, we have $%
\varphi _{N}\left( x\right) =0_{\mathbb{B}}^{N}.$ Since $\varphi _{N}\left(
x^{\ast _{N}}\right) =\varphi _{N}\left( x\right) ^{\ast _{N}}=\left( 0_{%
\mathbb{B}}^{N}\right) ^{\ast _{N}}=0_{\mathbb{B}}^{N},$ we write $x^{\ast
_{N}}\in Ker_{N}\left( \varphi _{N}\right) $ and hence $x\in \left(
Ker_{N}\left( \varphi _{N}\right) \right) ^{\ast _{N}}.$ This means that $%
Ker_{N}\left( \varphi _{N}\right) \subset \left( Ker_{N}\left( \varphi
_{N}\right) \right) ^{\ast _{N}}.$ The proof is completed.
\end{proof}

\begin{theorem}
Let $\mathbb{A}$ and $\mathbb{B}$ be $\ast _{N}-$algebras and $\varphi _{N}$
be a $\ast _{N}-$homomorphism from $\mathbb{A}$ to $\mathbb{B}$. Then, $%
\varphi _{N}\left( \mathbb{A}\right) $ is a $\ast _{N}-$subalgebra of $%
\mathbb{B}$.
\end{theorem}

\begin{proof}
We know that $\varphi _{N}\left( \mathbb{A}\right) $ is an $N-$subalgebra of 
$\mathbb{B}$ by Theorem 4. Then, it is enough to show that $\varphi
_{N}\left( \mathbb{A}\right) $ is $N-$self adjoint, that is, $\left( \varphi
_{N}\left( \mathbb{A}\right) \right) ^{\ast _{N}}=\varphi _{N}\left( \mathbb{%
A}\right) $ holds.

Let $y\in \left( \varphi _{N}\left( \mathbb{A}\right) \right) ^{\ast _{N}}.$
Then, $y^{\ast _{N}}\in \varphi _{N}\left( \mathbb{A}\right) $ by definition
of $N-$involution. In this case, there exists an element $x\in \mathbb{A}$
such that $\varphi _{N}\left( x\right) =y^{\ast _{N}}.$ Hence, $\varphi
_{N}\left( x\right) ^{\ast _{N}}=\varphi _{N}\left( x^{\ast _{N}}\right) =y.$
This implies that $y\in \varphi _{N}\left( \mathbb{A}\right) $ and so $%
\left( \varphi _{N}\left( \mathbb{A}\right) \right) ^{\ast _{N}}\subset
\varphi _{N}\left( \mathbb{A}\right) .$

Conversely, let $y\in \varphi _{N}\left( \mathbb{A}\right) .$ Then, there
exists an element $x\in \mathbb{A}$ such that $\varphi _{N}\left( x\right)
=y.$ So since $\varphi _{N}\left( x\right) ^{\ast _{N}}=\varphi _{N}\left(
x^{\ast _{N}}\right) =y^{\ast _{N}},$ we write $y^{\ast _{N}}\in \varphi
_{N}\left( \mathbb{A}\right) $ and hence $y\in \left( \varphi _{N}\left( 
\mathbb{A}\right) \right) ^{\ast _{N}}.$ This means that $\varphi _{N}\left( 
\mathbb{A}\right) \subset \left( \varphi _{N}\left( \mathbb{A}\right)
\right) ^{\ast _{N}}.$ The proof is completed.
\end{proof}

\begin{example}
Consider the $C^{\ast _{N}}-$\textit{algebra }$C_{N}\left( \Omega \right) $
given in Example 12 for $\Omega =\left\{ z\in 
\mathbb{C}
\left( N\right) :\overset{..}{\parallel }z\overset{..}{\parallel }\overset{..%
}{\leq }\frac{\overset{..}{1}}{\overset{..}{2}}\beta \right\} $ and the $%
\ast _{N}-$\textit{homomorphism} $\varphi _{N}$ given in Example 16. $%
Ker_{N}\left( \varphi _{N}\right) =\left\{ f\in C_{N}\left( \Omega \right)
:f\left( 0_{%
\mathbb{C}
\left( N\right) }\right) =0_{%
\mathbb{C}
\left( N\right) }\right\} $ is an $N-$self adjoint ideal of $C_{N}\left(
\Omega \right) $ and $\varphi _{N}\left( C_{N}\left( \Omega \right) \right) $
is an $\ast _{N}-$subalgebra of $%
\mathbb{C}
\left( N\right) .$
\end{example}

\section{Conclusion and Future Work}

We have introduced a new version of $C^{\ast }-$algebras which called a
non-Newtonian $C^{\ast }-$algebra. It is a relatively new addition to the
existing literature and generalizes known $C^{\ast }-$algebras.

Each choice of specific isomorphisms for $\alpha $ and $\beta $ determines a 
$\ast -$calculus \cite{11}. Geometric calculus obtained by choosing $\alpha
=I$ and $\beta =\exp $ \cite{11} is one of the most popular $\ast -$calculi
and has a wide variety of useful applications. In this direction, the
concept of a non-Newtonian $C^{\ast }-$algebra can be further enrich by
introducing the idea of geometric $C^{\ast }-$algebra wherein we replace $%
\ast -$points $\left( \overset{.}{a},\overset{..}{b}\right) $ in the $\ast -$%
complex field $%
\mathbb{C}
\left( N\right) $ by the $\ast -$points $\left( a,e^{b}\right) $ with
respect to the geometric calculus. Also, the notions of spectrum and
positive element in the sense of geometric calculus can be defined and thus
fixed point theorems can be obtained by contructing non-Newtonian $C^{\ast
}- $algebra valued metric spaces.


\begin{thebibliography}{99}
\bibitem{1} Grossman M. and Katz R., Non-Newtonian Calculus, Lee Press,
Pigeon Cove, MA, 1972.

\bibitem{2} Tekin S. and Ba\c{s}ar F., Certain sequence spaces over the
non-Newtonian complex field, Abstract and Applied Analysis, Volume 2013,
Article ID 739319.

\bibitem{3} \c{C}akmak A. F. and Ba\c{s}ar F., Some new results on sequence
spaces with respect to non-Newtonian calculus, Journal of Inequalities and
Applications, 2012, 2012:228.

\bibitem{4} \c{C}akmak A. F. and Ba\c{s}ar F., Certain spaces of functions
over the field of non-Newtonian complex numbers, Abstract and Applied
Analysis, Volume 2014, Article ID 236124.

\bibitem{5} Kadak U., Determination of the K\"{o}the-Toeplitz duals over the
non-Newtonian complex field, The Scientific World Journal, Volume 2014,
Article ID 438924.

\bibitem{6} Duyar C., Sa\u{g}\i r B. and O\u{g}ur, O., Some basic
topological properties on non-Newtonian real line, Brisith Journal of
Mathematics \& Computer Science, 9(4), 300-307, 2015.

\bibitem{7} Duyar C. and Erdo\u{g}an M., On non-Newtonian real number
series, IOSR Journal of Mathematics, 12(6), 34-48, 2016.

\bibitem{8} Kiri\c{s}ci M., Topological structures of non-Newtonian metric
spaces, Electronic Journal of Mathematical Analysis and Applications, 5(2),
156-169, 2017.

\bibitem{9} Murphy, G. J. $C^{\ast }-$Algebras and Operator Theory, Academic
press, Boston, 1990.

\bibitem{10} Kadak U. and Efe H., The construction of Hilbert spaces over
the non-Newtonian field, International Journal of Analysis, Volume 2014,
Article ID 746059.

\bibitem{11} Grossman M., An introduction to non-Newtonian calculus,
International Journal of Mathematical Educational in Science and Technology,
10(4), 525-528, 1979.
\end{thebibliography}
\end{document}